\documentclass[reqno,12pt,twosided,a4paper]{amsart}
\usepackage{amssymb,amsfonts,latexsym,mathrsfs,amsmath}
\usepackage[all]{xy}
\usepackage{array}

\usepackage{amstext}
\usepackage{amssymb}
\usepackage{amsfonts}
\usepackage{amsthm}
\usepackage{amsopn}
\usepackage{amsbsy}
\usepackage{layout}
\usepackage{color}

\newtheorem{theorem}{Theorem}[section]
\newtheorem{lemma}[theorem]{Lemma}
\newtheorem{proposition}[theorem]{Proposition}

\theoremstyle{definition}
\newtheorem{definition}[theorem]{Definition}

\numberwithin{equation}{section}

\newtheorem{question}[theorem]{Question}
\newtheorem{conjecture}[theorem]{Conjecture}
\newtheorem{remark}[theorem]{Remark}

\newcommand{\Z}{\mathbb{Z}}
\newcommand{\Q}{\mathbb{Q}}

\newcommand{\N}{\mathbb{N}}

\newcommand{\Ow}{\mathcal{O}}

\newcommand{\ot}{\otimes}

\newcommand{\ti}{\widetilde}
\newcommand{\ga}{\gamma}
\newcommand{\Ga}{\Gamma}

\newcommand{\A}{{\mathcal{A}}}
\newcommand{\ra}{\rightarrow}
\newcommand{\wb}{\overline}

\newcommand{\alH}{^{\alpha}H}
\newcommand{\alHal}{^{\alpha}H^{\alpha^{-1}}}
\newcommand{\tiS}{\ti{S}}
\newcommand{\YD}{\mathcal{YD}}
\newcommand{\cdota}{\cdot_{\alpha}}
\newcommand{\tb}{\textbf}
\newcommand{\ep}{\epsilon}
\newcommand{\la}{\lambda}
\newcommand{\allH}{\!^{\alpha}H}
\newcommand{\Aa}{\mathbb{A}}
\newcommand{\Hom}{\text{Hom}}
\newcommand{\End}{\text{End}}
\newcommand{\T}{\text{T}}
\newcommand{\B}{\text{B}}

\author{Ehud Meir}
\address{Institute of Mathematics, University of Aberdeen, Aberdeen, UK}
\email{meirehud@gmail.com}
\title[cocycles and invariant theory]{Hopf cocycle deformations and invariant theory}
\begin{document}
\maketitle
\begin{abstract}
For a given finite dimensional Hopf algebra $H$
we describe the set of all equivalence classes of cocycle deformations of $H$ as an affine variety, using methods of geometric invariant theory. 
We show how our results specialize to the Universal Coefficients Theorem in the case of a group algebra, 
and we also give examples from other families of Hopf algebras, including dual group algebras and Bosonizations of Nichols algebras. 
In particular, we use the methods developed here to classify the cocycle deformations of a dual pointed Hopf algebra associated to the symmetric group on three letters.
We also give an example of a cocycle deformation over a dual group algebra, which has only rational invariants, but which is not definable over the rational field.
This differs from the case of group algebras, in which every 2-cocycle is equivalent to one which is definable by its invariants.
\end{abstract}
\begin{section}{Introduction}
Let $H$ be a finite dimensional Hopf algebra defined over a field $K$. 
A \textit{Hopf 2-cocycle} (or simply 2-cocycle) on $H$ is a convolution invertible map $\alpha:H\ot H\to K$ which satisfy a certain associativity condition.
A \textit{cocycle deformation} of $H$ is then an associative $H$-comodule algebra of the form $^{\alpha}H$, where $\alpha$ is some 2-cocycle.
This algebra has the underlying vector space of $H$, and the multiplication in $\alH$ is given by the formula
$$x\cdota y = \alpha(x_1,y_1)x_2y_2.$$ The associativity of this algebra is equivalent to the associativity equations $\alpha$ satisfies. 
The coaction of $H$, $\rho:\allH\to \allH\ot H$ is given by the coproduct of $H$.
We will identify henceforth between 2-cocycles and the cocycle deformations they define. Thus, two 2-cocycles $\alpha$ and $\alpha'$ will be considered equivalent if and only if $^{\alpha'}H\cong \allH$ as $H$-comodule algebras.

2-cocycles appear abundantly in the theory of Hopf algebras. The algebra $\alH$, for example, can be seen as Hopf-Galois extension of the ground field $K$ (see Theorem 3.8. in \cite{Montgomery2}).
From the non-commutative geometry point of view, they can be thought of as principal bundles over a point 
(which, in the non-commutative case have some very nontrivial examples, see \cite{Schneider}).
From the categorical perspective such a structure is equivalent to a fiber functor on the category of $H$-comodules.

Hopf 2-cocycles can also be used to deform the multiplication of $H$ to form a new Hopf algebra, $^{\alpha}H^{\alpha^{-1}}$. The Hopf algebras of the form $^{\alpha}H^{\alpha^{-1}}$ are exactly the Hopf algebras whose category of comodules is equivalent to that of $H$. 
Such deformations appear abundantly in the classification of both semisimple and non semisimple Hopf algebras.

On the semisimple side, classification results are achieved by comparing the categories of modules or the categories of comodules of different Hopf algebras, and using the theory of Deligne of symmetric monoidal categories (see \cite{EG3}). 
On the non-semisimple side, classification results are achieved mostly by using the lifting method (see for example \cite{AndSch}). This method deals with the classification of all Hopf algebras $H$ such that $grH \cong \B(V)\# H_0$, where $H_0$ is a semisimple Hopf algebra, usually a group algebra or a dual group algebra, $\B(V)$ is a Nichols algebra, and $grH$ is the graded Hopf algebra associated to the coradical filtration of $H$. It turns out that in many cases all such Hopf algebras arise from $\B(V)\# H_0$ via a cocycle deformation. In \cite{AndSch} Andruskiewitsch and Schneider classified all Hopf algebras whose cordical is an abelian group algebra, under some restrictions on the order of this group. Masuoka showed later in \cite{Masuoka2} that all such Hopf algebras are cocycle deformations of their associated graded Hopf algebras. This classification was completed by Angiono and Garcia Iglesias in \cite{AnG}, where it is also shown that all Hopf algebras whose coradical is an abelian group algebra are cocycle deformations of their associated graded Hopf algebras. For more classification results of non-semisimple Hopf algebras in which cocycle deformations appear, see \cite{AAIMV}, \cite{FGM}, \cite{GarMa}, \cite{GruMa} and \cite{AndVay1}.

The importance of cocycle deformations raises the natural question of classification of such structures, up to an appropriate isomorphism.
In case the Hopf algebra $H$ is semisimple, Ocneanu rigidity tells us that there are only finitely many 2-cocycles up to equivalence (see \cite{ENO}).
For a large class of Hopf algebras, the group theoretical Hopf algebras, such cocycles can be classified explicitly by group theoretical data (see \cite{ENO} and \cite{GalPla}).
The general classification of equivalence classes of 2-cocycles is in general open.

In the specific case where $H=KG$ is a group algebra of a finite group $G$, 
an equivalence class of a 2-cocycle on $H$ is the same thing as an element in the second cohomology group
$H^2(G,K^{\times})$. If $K$ is algebraically closed and of characteristic zero, the Universal Coefficients Theorem gives us an isomorphism $$\Phi:H^2(G,K^{\times})\cong \Hom(H_2(G,\Z),K^{\times}).$$
This means that for every $c\in H_2(G,\Z)$ and every $\alpha\in H^2(G,K^{\times})$ we can view $\Phi([\alpha])(c)\in K^{\times}$ as a \textit{scalar invariant }
of the cohomology class of $\alpha$. The above isomorphism tells us that this set of scalar invariants is a \textit{complete set} of scalar invariants for the cocycle $\alpha$.

Another place in which scalar invariants of 2-cocycles appear is in the classification results of non-semisimple Hopf algebras. Unlike the case of group algebras, the resulting invariants now vary continuously. See for example \cite{AndSch} and \cite{AnG}.

The goal of the present paper is to study this classification problem, for a general finite dimensional Hopf algebra over an algebraically closed field of characteristic zero, from a geometric point of view.
This will continue the study done in \cite{Meir2} and \cite{DKS} where geometric invariant theory was applied to study finite dimensional semisimple Hopf algebras.
We will show that for a given Hopf algebra $H$ over an algebraically closed field $K$ of characteristic $0$ the set of all equivalence classes of cocycle deformations $X_H$ has a natural structure of an affine algebraic variety. 
We will thus think of $X_H$ as the moduli space of all the equivalence classes of 2-cocycles on $H$. 
We will use methods of geometric invariant theory to construct this variety as a quotient of an affine variety by the action of some reductive group. 
The invariants and the variety that we will get here will generalize the invariants one receives for cocycles on group algebras from the Universal Coefficients Theorem on the one hand, and the continuous invariants which appear in the classification of non-semisimple Hopf algebras on the other hand. We will give examples for both.

To state the result, recall first that the algebra $\alH$ is an $H$-comodule algebra which is a Hopf-Galois extension of the ground field $K$ (see \cite{Montgomery2}). 
Such $H$-comodule algebras are characterized by the fact that the map $M:\allH\ot\allH\to\allH\ot H$ $x\ot y\mapsto x\cdota y_1\ot y_2$ is invertible (see Proposition \ref{prop:FormofW}).
Let us write $T:\allH\ot H\to \allH\ot\allH$ for the linear inverse of this map. For $h\in H$ we then write $T_h:\allH\to \allH\ot \allH$ $y\mapsto T(y\ot h)$.
For elements $h(1),\ldots, h(l)$ we write $T(h(1),\ldots, h(l)): W\to W^{\ot (l+1)}$ for the composition \begin{equation} T(h(1),\ldots, h(l)) = (Id_W^{\ot l-1}\ot T_{h(l)})\cdots(Id_W\ot T_{h(2)})T_{h(1)}.\end{equation}
For $f\in H^*$ we write $A_f:\allH\to \allH$ $x\mapsto x_1f(x_2)$ for the action of $H^*$ induced from the coaction of $H$.
For a permutation $\sigma\in S_{l+1}$ we write $L_{\sigma}:(\allH)^{\ot l+1}\to (\allH)^{\ot l+1}$ for the permutation of the tensor factors induced by $\sigma$.

The main result of this paper is the following theorem, which will be proven in Section \ref{sec:Invariants}:
\begin{theorem}\label{thm:main}
For every finite dimensional Hopf algebra $H$ the set $X_H$ of equivalence classes of cocycle deformations of $H$ has a natural structure of an affine variety, and we can therefore think of the elements in the coordinate ring $K[X_H]$ 
as a complete set of scalar invariants for 2-cocycles on $H$. The commutative algebra $K[X_H]$ has a set of generators of the following form:
$$c(l,\sigma,f,h(1),\ldots h(l)) := Tr_{W^{\ot l+1}}(A_fm^lL_{\sigma}T(h(1),\ldots h(l))),$$
where $l\in\mathbb{N}$, $f\in H^*$, $\sigma\in S_{l+1}$ and $h(i)\in H$. We call these elements the basic invariants of $\alH$. 
The relations these invariants satisfy are given explicitly in Section \ref{sec:Invariants}.
\end{theorem}
Notice that for the above theorem we do not require the Hopf algebra $H$ to be semisimple, even though the semisimplicity was a necessary condition to apply geometric invariant theory to Hopf algebras in \cite{DKS} and \cite{Meir2}.

The fact that we have a complete set of invariants gives us an indication about the fields of definition of $\alH$. 
We will prove in Section \ref{sec:Invariants} the following proposition (see Proposition \ref{prop:Galois}):
\begin{proposition}\label{prop:Galois-main}
Let $W=\allH$ be a cocycle deformation of $H$, defined over $K$.
Assume that $K/k$ is a Galois extension with Galois group $G$, and that $H$ has a $k$-form $H_k$ (i.e. $H\cong H_k\ot_k K$).
Let $L=k(c(l,\sigma,f,h(1),\ldots h(l)))\subseteq K$ be the subfield of $K$ generated by the basic invariants of $W$ (where we take here $h(i)\in H_k$ and $f\in H_k^*$). 
  Then for $\gamma\in G$ it holds that $^{\gamma}W\cong W$ if and only if $\gamma$ fixes the subfield $L$ of $K$ pointwise.
  In particular, a necessary condition for $W$ to be defined over a subfield $F\subseteq K$ is $L\subseteq F$.
 \end{proposition}
 
Theorem \ref{thm:main} describes the ring $K[X_H]$ in terms of generators and relations. 
In order to prove this theorem we will first describe, in Section \ref{sec:Variety}, the variety $X_H$ as a quotient of the form $Y_H/GL_n(K)$ where $Y_H$ is an affine variety and $n=dim(H)$.
We will then have an isomorphism $K[X_H]= K[Y_H]^{GL(W)}$, and we will calculate this ring of invariants explicitly in Section \ref{sec:Invariants}. 

Carrying out the analysis of the algebra of invariants for a given Hopf algebra $H$ can be quite difficult (the presentation which we will give here has in general infinitely many generators and infinitely many relations).
In order to overcome this problem we will combine certain results from \cite{Meir1}, where certain classification results for 2-cocycles were also obtained by using a categorical construction, 
and we will show some alternative ways to describe the variety $X_H$ and the ring $K[X_H]$, which will help us to simplify the calculations.
We shall explain how, in many examples, we can find a subvariety $Y'\subseteq Y_H$ and a subgroup $N\subseteq GL(W)$ which acts on $Y'$ such that the natural map $Y'/N\to Y/GL(W)$ is an isomorphism of varieties, or at least a bijection. 

In Section \ref{sec:Examples} we will give examples.
For group algebras we will show how the isomorphism arising from the Universal Coefficients Theorem can be seen in the framework of the variety $X_H$ presented here.
We will also describe the invariants one receives for 2-cocycles on the dual group algebra $K[G]$, and their relation to invariants of cocycles on group algebras. 
As an application, we will give an example of a group $G$ and a cocycle deformation on $K[G]$, 
for which all the basic invariants (with respect to the canonical basis of $K[G]$) are contained in $\Q$, but for which the cocycle itself is not definable over $\Q$. This shows that the necessary condition from Proposition \ref{prop:Galois-main} is not sufficient.
This is in contrary to the case of group algebras, where every cocycle is definable over the extension of $\Q$ generated by its scalar invariants (see \cite{AHN}). 
See also \cite{AK} for another approach to cocycle deformations and the possible ways to define them over subfields, using their graded identities.

In the non-semisimple realm we will study algebras of the form $H=H_0\# \B(V)$, the Radford biproduct, or Bosonizations, of a semisimple Hopf algebra $H_0$ and the Nichols algebra of a
vector space $V\in \!_{H_{0}}^{H_0}\YD$. We will concentrate on the case where $H_0$ is either a group algebra or a dual group algebra. 
We will consider the Taft Hopf algebras, and the Hopf algebras in which $H_0=KS_3$ or $K[S_3]$ and $V = (\Ow^3_2,-1)$ in the terminology of \cite{IM}.
We will use the methods developed here to show that $X_H$ is $\Aa^2$ in case $H_0=KS_3$ and $\Aa^3$ (up to a bijective correspondence) in case $H_0=K[S_3]$. 
The first case appears in \cite{IM}. For the second case, the resulting double-twisted Hopf algebras $\alHal$ appear in \cite{AndVay1}.

In all the examples checked so far the resulting space $X_H$ is a disjoint union of affine spaces. In Section \ref{sec:Questions} we will formulate a conjecture saying that this is always the case.
We will also raise the question about the possibility to reconstruct $\alH$ from its invariants. 
\end{section} 

\begin{section}{Preliminaries}\label{sec:Prelim}
\subsection{Hopf algebras and cocycle deformations}
Throughout this paper $H$ will be a finite dimensional Hopf algebra of dimension $n$ defined over an algebraically closed field $K$ of characteristic zero. 
We make these assumptions about the field to simplify the application of geometric invariant theory to our situation. 
We revise here some known facts about Hopf algebras and their cocycle deformations, and refer the reader to \cite{Montgomery} and the introduction of \cite{AndSch} for further reading. 
We will use here the Sweedler notation for the comultiplication in $H$:
\begin{equation}\Delta(x) = x_1\ot x_2.\end{equation} The counit of $H$ will be denoted by $\epsilon$.

Our main object of study in this paper will be cocycle deformations of $H$. 
We recall that a \textit{Hopf 2-cocycle} (or just 2-cocycle) on $H$ is a convolution invertible map $\alpha:H\ot H\to K$ which satisfies the associativity condition:
\begin{equation}\forall x,y,z\in H:\quad \alpha(x_1,y_1)\alpha(x_2y_2,z) = \alpha(y_1,z_1)\alpha(x,y_2z_2).\end{equation}
We will also assume here that our 2-cocycles satisfy the unity condition 
\begin{equation}\alpha(1,x)=\alpha(x,1)=\epsilon(x).\end{equation}
This assumption can be made because every 2-cocycle is equivalent to a 2-cocycle which satisfies this condition.
2-cocycles enable us to define new algebras. We define an algebra $\alH$ which has the underlying vector space $H$ and in which the multiplication is given by the formula:
\begin{equation}x\cdota y = \alpha(x_1,y_1)x_2y_2.\end{equation}
The conditions on $\alpha$ assure that the multiplication in $\alH$ is associative, and that $1\in\allH$ remains a unit for the twisted multiplication.
The algebra $\alH$ has a richer structure of an $H$-comodule algebra. That is: by identifying $\alH$ with $H$ as vector spaces,
the map $\Delta$ induces an algebra map $\rho:\allH\to \allH\ot H$ which furnishes an $H$-comodule structure on $\alH$. 

 
The algebra $\alH$ is an $H$-comodule algebra, but in general it is not a Hopf algebra by itself.
Indeed, the counit of a Hopf algebra provides us with an algebra homomorphism from the Hopf algebra into the ground field $K$, and $\alH$ will admit such a homomorphism if and only if $\alH\cong H$ as $H$-comodule algebras. 
However, one can construct a double-twisted Hopf algebra $\alHal$. This Hopf algebra has the underlying vector space $H$, has the same coalgebra structure as $H$, and a two-sided twisted multiplication, given by the formula
\begin{equation}x\cdot y  = \alpha(x_1,y_1)\alpha^{-1}(x_3,y_3)x_2y_2.\end{equation}
The unit and counit in this new Hopf algebra are the same as those in $H$. 
The antipode in this new algebra is more complicated, and is given by the formula
\begin{equation}S^{\alpha}(x) = \gamma(x_1)S(x_2)\gamma^{-1}(x_3)\end{equation} where $\gamma\in H^*$ is given by $\gamma(x) = \alpha(x_1,S(x_2))$.
We will prove later that the element $\gamma$ is really invertible, so this is well defined.
The notion of a 2-cocycle on a Hopf algebra is dual to the notion of a Drinfeld twist on a Hopf algebra. In other words, a 2-cocycle on $H$ is the same as a Drinfeld twist in $H^*\ot H^*$.

It is possible that different cocycles $\alpha$ and $\alpha'$ will define isomorphic cocycle deformations.
To see when this happens notice first that an isomorphism $\alH\to\! ^{\alpha'}H$ is in particular an isomorphism of $H$-comodules, and will therefore be of the form $x\mapsto \nu(x_1)x_2$
for some invertible $\nu\in H^*$ (this follows from the fact that $H$-comodules are the same as $H^*$-modules, and that $H$ is isomorphic with $H^*$ as $H^*$-modules).
It then follows that this map will be an isomorphism of algebras if and only if the equation
\begin{equation}\nu(x_1)\nu(y_1)\alpha'(x_2,y_2)\nu^{-1}(x_3y_3) = \alpha(x,y)\end{equation} holds.
This also gives us a description of the automorphism group of $\alH$ as an $H$-comodule algebra. 
Indeed, by the above formula we see that an invertible element $\nu\in H^*$ will define 
an automorphism $\alH\to \allH$ if and only if $\nu:\!\alHal\to K$ is an algebra homomorphism.
We conclude this discussion in the following lemma:
\begin{lemma}\label{lemma:FiniteAutomorphisms}
The automorphism group of $\alH$ as an $H$-comodule algebra is canonically isomorphic with the group of characters of the Hopf algebra $\alHal$.
In particular, since this algebra is finite dimensional, this group is finite.
\end{lemma}

We next show that the invertibility of $\alpha$ is equivalent to the invertibility of a certain element in $H^*$.
Let $\gamma\in H^*$ be defined by \begin{equation}\gamma(x) = \alpha(x_1,S(x_2)).\end{equation} Notice that it holds that $x_1\cdota S(x_2) = \alpha(x_1,S(x_4))x_2S(x_3) = \gamma(x)$. 
We claim the following:
\begin{lemma} The invertibility of $\alpha\in (H\ot H)^*$ is equivalent to the invertibility of $\gamma\in H^*$.
\end{lemma}
\begin{proof}
Let $x,y\in H$.
consider the product $x_1\cdota y_1\cdota S(y_2)\cdota S(x_2)$.
Using the associativity of $\alpha$, we will write this element in two different forms. On the one hand, starting with the multiplication of $y_1$ and $S(y_2)$ we see that this element is equal to $\gamma(x)\gamma(y)$.
On the other hand, we also have 
$$x_1\cdota y_1\cdota S(y_2)\cdota S(x_2) = \alpha(x_1,y_1)\alpha(S(y_4),S(x_4))\cdot$$ \begin{equation}x_2y_2\cdota S(x_3y_3) =  \alpha(x_1,y_1)\gamma(x_2y_2)\alpha(S(y_3),S(x_3)).\end{equation}
In other words, we get the following equation in $H^*\ot H^*$:
\begin{equation}\alpha\cdot \Delta(\gamma)\cdot ((S\ot S)(\alpha^{op})) = \gamma\ot \gamma.\end{equation}
This means that if $\gamma$ is invertible then $\alpha$ is invertible, and its explicit inverse is given by 
\begin{equation}\label{eq:Inverse}\alpha^{-1}=\Delta(\gamma)((S\ot S)(\alpha^{op}))(\gamma^{-1}\ot\gamma^{-1}).
\end{equation}
 Moving $\Delta(\gamma)$ to the other side of the equation and writing everything as functionals on $H$ gives us
	\begin{equation}\label{eq:Inverse2}
		\gamma^{-1}(x_1y_1)\alpha^{-1}(x_2,y_2) = \alpha(S(y_1),S(x_1))\gamma^{-1}(x_2)\gamma^{-1}(y_2)
	\end{equation}
for every $x,y\in H$.

On the other hand, assume that $\alpha$ is invertible. We write $\beta=\alpha^{-1}$, $\alpha = \alpha^{(1)}\ot \alpha^{(2)}$ and similarly for $\beta$.
The 2-cocycle condition for $\alpha$ reads $(\alpha\ot 1)(\Delta(\alpha^{(1)})\ot \alpha^{(2)}) = (1\ot \alpha)(\alpha^{(1)}\ot\Delta(\alpha^{(2)}))$.
Multiplying by the relevant inverses from both sides, we get the equation
\begin{equation}(1\ot \beta)(\alpha\ot 1) = (\alpha^{(1)}\ot\Delta(\alpha^{(2)}))(\Delta(\beta^{(1)})\ot \beta^{(2)}).\end{equation}
By applying $S$ to the middle factor and multiplying all the three tensors together we get the equation
\begin{equation}\alpha^{(1)}S(\alpha^{(2)})S(\beta^{(1)})\beta^{(2)} = 1.\end{equation}
But this equation translates to $\gamma S(\beta^{(1)})\beta^{(2)} = 1$, so $\gamma$ is invertible as desired. 
\end{proof}
Next, we define a twisted antipode $\tiS:\allH\to\allH$ by \begin{equation}\tiS(x) = S(x_1)\gamma^{-1}(x_2).\end{equation}
We claim the following:
\begin{lemma}\label{lemma:TwistedAntipode}
In $\alH$ it holds that $x_1\cdota \tiS(x_2) = \tiS(x_1)\cdota x_2 = \epsilon(x)$ and $\tiS(x)\cdota \tiS(y) = \alpha^{-1}(y_2,x_2)\ti{S}(y_1x_1)$
\end{lemma}
\begin{proof}
By the last lemma we know that $\gamma^{-1}(x) = \alpha^{-1}(S(x_1),x_2)$.
We calculate: $$\tiS(x_1)\cdota x_2 = \gamma^{-1}(x_3)\alpha(S(x_2),x_4)S(x_1)x_5 = $$ \begin{equation}
\alpha^{-1}(S(x_3),x_4)\alpha(S(x_2),x_5)S(x_1)x_6 = S(x_1)x_2 = \epsilon(x)\end{equation}
where for the last equation we have used the fact that $\alpha^{-1}$ and $\alpha$ multiply to $\epsilon\ot\epsilon$.
For the second equation we get
$$x_1\cdota \tiS(x_2) = \alpha(x_1,S(x_4))x_2S(x_3)\gamma^{-1}(x_5) = $$ \begin{equation}\alpha(x_1,S(x_2))\gamma^{-1}(x_3) = \gamma(x_1)\gamma^{-1}(x_2) = \epsilon(x)\end{equation} as desired.
For the last equality, we calculate, using Equation \ref{eq:Inverse2}
$$\tiS(x)\cdota \tiS(y) = \alpha(S(x_2),S(y_2))S(x_1)S(y_1)\gamma^{-1}(x_3)\gamma^{-1}(y_3)= $$
\begin{equation} S(y_1x_1)\gamma^{-1}(y_2x_2)\alpha^{-1}(y_3,x_3) = \tiS(y_1x_1)\alpha^{-1}(y_2,x_2)\end{equation} as desired.
\end{proof}

If $W$ is a comodule algebra which is of the form $\alH$, we can think of the choice of the 2-cocycle $\alpha$ as a choice of coordinates for $W$.
Indeed, the choice of the 2-cocycle is equivalent to the choice of an isomorphism $W\cong H$ as $H$-comodules.
We would like to give here also a ``coordinate free'' version of this structure, which will help us later on in the application of geometric invariant theory.

Assume now that $W$ is an $H$-comodule algebra. We denote the multiplication in $W$ by concatenation or by $\cdot$ and the coaction of $H$ by $\rho:W\to W\ot H$, $w\mapsto w_1\ot w_2$.
The map \newpage $$M:W\ot W\to W\ot H$$ \begin{equation}x\ot y\mapsto xy_1\ot y_2\end{equation} will play a prominent role in what follows. We have the following proposition (see Theorem 3.8 in \cite{Montgomery2} and references therein. See also \cite{Schauenburg1});
\begin{proposition}\label{prop:FormofW}
A finite dimensional $H$-comodule algebra $W$ is of the form $\alH$ if and only if $W\neq 0$ and the map $M$ defined above is invertible.
\end{proposition}
The following lemma gives a convenient criterion to the invertibility of $M$. To state it, notice that if $M$ is invertible with an inverse $T$, then the map $\ti{T}:H\to W\ot W$ given by $\ti{T} = T(1\ot h)$ considered as a map from $H$ to $W^{op}\ot W$ is an algebra homomorphism. It turns out that the invertibility of $M$ can be detected by considering this map. 
\begin{lemma}\label{lemma:Minv}	
The map $M$ is invertible if and only if there exists a homomorphism of algebras $$\ti{T}:H\to W^{op}\ot W$$
for which the composition $$H\stackrel{\ti{T}}{\to}W\ot W\stackrel{M}{\to}W\ot H$$ is equal to $h\mapsto 1\ot h$,
and the composition $$W\stackrel{\rho}{\to}W\ot H\stackrel{1\ot\ti{T}}{\to}W\ot W\ot W\stackrel{m_W\ot 1}{\to}W\ot W$$ is equal to $w\mapsto 1\ot w$.
Moreover, it is enough to check the equality of these compositions on some multiplicative-generating sets for $H$ and for $W$.
\end{lemma}
\begin{proof}
Assume first that $W\cong \allH$. For convenience, assume further that $W=\allH$. We define $\ti{T}(h) = T(1\ot h)$ where $T$ is the linear inverse of $M$. 
In this case the map $\ti{T}$ can be calculated explicitly in terms of the algebra $\alH$. Indeed, we get $\ti{T}(h) = \tiS(h_1)\ot h_2\in W\ot W$. 
This map is multiplicative when considered as a map $H\to W^{op}\ot W$ since
$$\ti{T}(x)\ti{T}(y) = \tiS(y_1)\cdota \tiS(x_1)\ot x_2\cdota y_2 =
$$
\begin{equation}\tiS(x_1y_1)\alpha^{-1}(x_2,y_2)\alpha(x_3,y_3)\ot x_4y_4 =  \tiS(x_1y_1)\ot x_2y_2 = \ti{T}(xy).\end{equation}

In the other direction, assume that such a map $\ti{T}$ exists, and extend it to a map $T:W\ot H\to W\ot W$ by $w\ot h\mapsto (w\ot 1)\cdot \ti{T}(h)$.
Then the conditions of the lemma imply that $T$ is the inverse of $M$. This follows from the fact that both maps are $W$-linear where $W$ acts on the left tensor factor,
and the compositions of $M$ and $T$ in both directions are the identity when evaluated on a generating subset (as $W$-modules) of these $W$-modules. 

To see why it is enough to check the conditions of the lemma on a generating set (in the multiplicative sense), we show that the set of elements $h\in H$ for which $M\ti{T}(h)=1\ot h$ is closed under multiplication.
In a similar way, we show that the set of elements $w\in W$ for which $(m_W\ot 1)(1\ot\ti{T})\rho(w) = 1\ot w$ is closed under multiplication.
Assume then that $M\ti{T}(x) = 1\ot x$ and that $M\ti{T}(y) = 1\ot y$.
Write $\ti{T}(x) = \sum_i a_i\ot b_i$ and $\ti{T}(y) = \sum_j c_j\ot d_j$.
This means that $$\sum_i a_i(b_i)_1\ot (b_i)_2= 1\ot x\text{ and } $$ \begin{equation}\sum_j c_j(d_j)_1\ot (d_j)_2 = 1\ot y.\end{equation}
We then have $$M\ti{T}(xy) = M(\ti{T}(x)\ti{T}(y)) = M(\sum_{i,j}c_ja_i\ot b_id_j) = $$ \begin{equation}\sum_{i,j}c_ja_i(b_i)_1(d_j)_1\ot (b_i)_2(d_j)_2 = 
\sum_j c_j(d_j)_1\ot x(d_j)_2 = 1\ot xy.\end{equation}
In the other direction, we use the fact that $(1\ot \ti{T})\rho:W\to W\ot W^{op}\ot W$ is an algebra map. 
Then if $w,w'\in W$ satisfy the condition of the lemma and we write $(1\ot \ti{T})\rho(w)= \sum_i a_i\ot b_i\ot c_i$ and $(1\ot \ti{T})\rho(w') = \sum_j r_j\ot s_j\ot t_j$, then it holds that 
$\sum_i a_ib_i\ot c_i = 1\ot w$ and $\sum_j r_js_j\ot t_j = 1\ot w'$.
We then have that $$(m_W\ot 1)(1\ot \ti{T})\rho(w\cdot w') = \sum_{i,j}a_ir_js_jb_i\ot c_it_j = $$ \begin{equation}\sum_i a_ib_i\ot c_iw' = 1\ot ww'\end{equation}
and we are done.
\end{proof}

The coordinate free perspective also enables us to construct the algebra $\alHal$ categorically in terms of the algebra $\alH$: 
\begin{lemma}[see \cite{Schauenburg1}]\label{lemma:Schauenburg} The subspace of $H$-coinvariants in $\alH\ot\allH$ is a subalgebra under the multiplication in $\alH\ot(\allH)^{op}$. 
It is isomorphic to the algebra $\alHal$, and the map $H\to \allH\ot\allH$ which sends $x$ to $x_1\ot \tiS(x_2)$ is an isomorphism of coalgebras between $H$ and the algebra of coinvariants.
\end{lemma}
\begin{proof} For a proof of the first statement see \cite{Schauenburg1}. The second statement is a direct calculation. 
\end{proof}

\begin{remark} In the proof of Lemma \ref{lemma:Minv} we have seen that an explicit formula for the inverse of $M$ is given by $T(x\ot y) = x\cdota \tiS(y_1)\ot y_2$.
Notice in particular that the map $M$ is both a $W$-module map and an $H$-comodule map, where $W$ acts from the left on the left tensor factor on both sides, and $H$ coacts from the right on the right tensor factor of $W$ in $W\ot W$ and on
the tensor factor $H$ in $W\ot H$. 
\end{remark}

The map $T$ can also be used to define on $\alH$ a canonical structure of a Yetter-Drinfeld module. We define a right action of $H$ on $\alH$ in the following way:
\begin{equation}\forall x\in \allH,h\in H\quad x\cdot h = mL_{(23)}T(1\ot h)\ot x.\end{equation}
The map $m$ is the multiplication on $\alH$ and $L_{(23)}:(\allH)^{\ot 3}\to (\allH)^{\ot 3}$ is given explicitly by $a\ot b\ot c\to a\ot c\ot b$.
Using the explicit formula for $T$ we obtained we see that $T(1\ot h) = \tiS(h_1)\ot h_2$. This implies that $x\cdot h = \tiS(h_1)\cdota x\cdota h_2$. 
Explicit calculation now shows us that $\rho(x\cdot h) = x_1\cdot h_2\ot S(h_1)x_2h_3$, which means that $\alH$ has a structure of a Yetter-Drinfeld module over $H$ indeed.

\subsection{Geometric invariant theory}
Let $K$ be an algebraically closed field of characteristic zero as before, and let $Y$ be an affine variety defined over $K$. 
This means that $Y$ can be thought of as the set of zeros of a collection of polynomials $\{f_1,\ldots f_m\}\subseteq K[y_1,\ldots y_n]$. 
We write $K[Y]:=K[y_1,\ldots y_n]/(f_1,\ldots,f_m)$ and think of this ring as the ring of polynomial functions on $Y$. We do assume here that $(f_1,\ldots f_m)$ is a radical ideal. 
Even if it is not the case, we can still define $Y$ as the set of zeros of $f_1,\ldots, f_m$, but in the definition of $K[Y]$ we need to take the radical of the ideal $(f_1,\ldots, f_m)$.

Let $\Ga$ be a reductive algebraic group which acts on $Y$ algebraically. In this paper the group $\Ga$ will be a reductive subgroup of $GL_n$, 
the affine variety $Y$ will be the variety of all possible cocycle deformations of a given finite dimensional Hopf algebra $H$, and two points in $Y$ will define isomorphic cocycle deformations if and only if they lie in the same orbit of $\Ga$.
For this reason we would like to form the quotient space $Y/\Ga$. Every polynomial function $f\in K[Y]$ which is invariant under the induced action of $\Ga$, $g\cdot f (y) = f(g^{-1}y)$, can be thought of as a polynomial function on $Y/\Ga$.
The following central result from Geometric Invariant Theory (GIT) tells us when the other direction works as well (see \cite[Theorem 3.5]{Newstead}).
\begin{theorem}\label{Thm:GIT}
Let $\Ga$ and $Y$ be as above. Assume that all the orbits of $\Ga$ in $Y$ are closed. Then the orbit space $Y/\Ga$ is also an affine variety.
Moreover, we have an isomorphism $K[Y/\Ga]\cong K[Y]^{\Ga}$, and the natural map $Y\ra Y/\Ga$ corresponds to the inclusion of algebras $K[Y]^{\Ga}\ra K[Y]$.
We have a one to one correspondence between closed $\Ga$-stable subsets of $Y$ and closed subsets of $Y/\Ga$. 
Therefore, if $I\subseteq K[Y]$ is a radical $\Ga$-stable ideal of $Y$, then $I\neq 0$ if and only if $I^{\Ga}\neq 0$.
\end{theorem}
The map $Y\to Y/\Ga$ satisfies a universal property with respect to morphisms of varieties $Y\to X$ which are invariant on $\Ga$-orbits (see \cite{Newstead} for more details).

In practice, a lot of the invariant rings which we shall encounter will be difficult to calculate explicitly.
The following proposition will be useful for reducing the variety and the group acting on it.
\begin{proposition}[Reduction of acting group]\label{prop:Reduction}
Let $\Ga$ be a reductive algebraic group acting on an affine variety $Y$. Assume that $N<\Ga$ is a closed reductive subgroup and that $Y'\subseteq Y$ is a closed subvariety
such that the following conditions hold:
\begin{enumerate}
\item All the orbits of $\Ga$ in $Y$ are closed. 
\item The subvariety $Y'$ is stable under the action of $N$.
\item For every $\Ga$-orbit $T$ in $Y$, the intersection $T\cap Y'$ is an $N$-orbit. in particular, every $\Ga$-orbit in $Y$ intersects $Y'$ non-trivially.
\end{enumerate}
Then the restriction of functions from $Y$ to $Y'$ induces an injective ring homomorphism $:\Phi:K[Y/\Ga]\cong K[Y]^{\Ga}\to K[Y']^N\cong K[Y'/N]$ which induces a bijection $Y'/N\to Y/{\Ga}$.
\end{proposition}
\begin{proof}
Since a $\Ga$-invariant polynomial on $Y$ is in particular $N$-invariant, 
the restriction map \ $K[Y]\to K[Y']$ induces a ring homomorphism $\Phi:K[Y]^{\Ga}\to K[Y']^N$. The subgroup $N$ acts on $K[Y']$ since it acts on $Y'$, by the second assumption.
Since every $\Ga$-orbit in $Y$ intersects $Y'$ it follows that if the restriction of $f\in K[Y]^{\Ga}$ to $Y'$ is zero, then it is zero on all the $\Ga$-orbits in $Y$, 
and is therefore zero on $Y$. This implies that the map $\Phi$ is injective.

The orbits of $N$ in $Y'$ are the intersections of the orbits of $\Ga$ in $Y$ with $Y'$. 
Since the orbits of $\Ga$ in $Y$ are closed, the same is true for the orbits of $N$ in $Y'$ since they are an intersection of two closed subsets.
Since both groups are reductive we have affine quotient varieties $Y/\Ga$ and $Y'/N$, and isomorphisms $K[Y]^{\Ga}\cong K[Y/\Ga]$ and $K[Y']^N\cong K[Y'/N]$.
By the universal property of the variety $Y'/N$ the map $Y'\to Y\to Y/\Ga$ induces a map $Y'/N\to Y/\Ga$, for which $\Phi$ is the induced map on coordinate rings.
The conditions of the proposition imply that the induced map $Y'/N\to Y/G$ is bijective (this is a purely set-theoretical argument which does not use the additional structure of the groups and the varieties).
\end{proof}
\begin{remark}
It is possible that the map $\Phi$ from the lemma will be injective but not surjective.
Consider for example the variety $Y=\mathbb{A}^2 \backslash \{(0,y)|y\in K\} = \{(x,y)\in K^2|x\neq 0\}$ and the subvariety $Y'= \{(x,1)|x\neq 0\}\cup\{(1,0)\}$.
We define $\Ga=\mathbb{G}_m$ to be the multiplicative group of the field and we define an action of $\Ga$ on $Y$ by $t\cdot (x,y) = (tx,t^{-1}y)$. 
We define the subgroup $N$ to be the trivial group. Then it is easy to show that $\Ga,N,Y$ and $Y'$ satisfy the conditions of the lemma: 
all orbits of the action of $\Ga$ on $Y$ are of the form $O_c:=\{(x,y)|xy=c\}\cap Y$ for some $c\in K$ and are therefore closed.
Notice that the intersection with $Y$ is redundant for all $c\neq 0$, but not for $c=0$.
The subvariety $Y'$ is trivially stable under the action of $N$, and the intersection of $O_c$ with $Y'$ for $c\neq 0$ is $\{(c,1)\}$, and $O_0\cap Y' = \{(1,0)\}$.
However, $K[Y]^{\Ga} = K[xy]$ is a polynomial ring in one variable, while $K[Y']^N = K[x^{\pm 1}]\oplus K$ is strictly bigger then $K[Y]^{\Ga}$.
\end{remark}
\end{section}

\begin{section}{The variety of cocycle deformations}\label{sec:Variety}
 Let $H$ and $K$ be as before.
From Proposition \ref{prop:FormofW} we know that a cocycle deformation of $H$ is the same as an $H$-comodule algebra $W$ of dimension $n=dim(H)$, for which the map 
$$M:W\ot W\to W\ot H$$ \begin{equation}x\ot y\mapsto x\cdot y_1\ot y_2\end{equation} is invertible.
Take now $W=K^{n}$. We would like to describe the space of all possible cocycle deformation structures on $W$. 
For this, we start with the following affine space: $$\A_H = \Hom_K(W\ot W,W)\bigoplus \Hom_K(W\ot H,W\ot W)$$ \begin{equation}\bigoplus
\Hom_K(H^*\ot W,W)\end{equation}
Notice that the group $\Ga=GL(W)$ acts in a natural way on all the direct summands appearing in $\A$ by its diagonal action on $W$ and the trivial action on $H$. 
The group $\Ga$ therefore acts on $\A_H$ as well. 
We will write a point in $\A_H$ as $(m,T,A)$.
We will think of $m$ as the multiplication on $W$, on $A$ as the action of $H^*$ on $W$ (which contains the same information as a coaction $\rho:W\to W\ot H$) and on $T:W\ot H\to W\ot W$ as the inverse of the map $M$ defined above.
Of course, not every point in $\A_H$ will satisfy the necessary axioms for a cocycle deformation.
Let us denote by $Y_H\subseteq \A_H$ the subset of all points $(m,T,A)$ which do give on $W$ the structure of a cocycle deformation of $H$. 
We claim the following:
\begin{lemma}
The subset $Y_H$ is Zariski closed in $\A_H$ and is stable under the action of $\Ga$.
\end{lemma}
\begin{proof}
The conditions on the points in $Y_H$ are the following: 
\begin{enumerate}
 \item Associativity of $m$ :
 \begin{equation} m(m\ot Id_W) = m(Id_W\ot m):W\ot W\ot W\ra W.\end{equation} 
 \item Associativity of the action of $H^*$: 
 \begin{equation}A(Id_{H^*}\ot A) = A(m_{H^*}\ot Id_W):H^*\ot H^*\ot W\ra W.\end{equation}
 \item Compatibility between the action and the multiplication: 
 \begin{equation} m(A\ot A)(Id_{H^*}\ot\tau\ot Id_W)(\Delta_{H^*}\ot Id_{W\ot W})=\end{equation} \begin{equation*}A(Id_{H^*\ot W}\ot m):H^*\ot W\ot W\ra W\end{equation*}
 where $\tau:H^*\ot W\to W\ot H^*$ is the flip map.
 \item The map $T:W\ot H\to W\ot W$ is the linear inverse of the map $M$ described above. 
 \end{enumerate}
All the coefficients of the linear maps mentioned here can be written as polynomials in the coefficients of the linear maps $m,T$ and $A$.
As a result, $Y_H$ is Zariski closed.
Since the equations we have here are stable under the action of $\Ga$ (since this action respects compositions of linear maps), the subset $Y_H$ is also stable under the action of $\Ga$.
\end{proof}
The next two lemmas are crucial for the use of Theorem \ref{Thm:GIT}
\begin{lemma}
Two points in $Y_H$ determine isomorphic cocycle deformations if and only if they lie in the same $\Ga$-orbit. 
\end{lemma}
\begin{proof} 
	Let $g\in \Ga$. We consider it as a linear isomorphism $g:W\to W$.
We claim that $g(m,T,A) = (m',T',A')$ if and only if $g$ is an isomorphism between the cocycle deformation structure defined by $(m,T,A)$ and the one defined by $(m',T',A')$. 
Indeed, $g(m)= m'$ means that the diagram 
\[ 
\xymatrix{
W\ot W\ar[r]^-m\ar[d]^{g\ot g} & W\ar[d]^g \\
W\ot W\ar[r]^-{m'} & W
}
\]
commutes, and similar statements hold for $T$ and for $A$. This shows us that two points $(m,T,A)$ and $(m',T',A')$ are in the same orbit if and only if there exists an isomorphism between $W$ considered as a cocycle deformation via $(m,T,A)$ and $W$ considered as a cocycle deformation via $(m',T',A')$. We are done.
\end{proof}
We therefore want to classify all the orbits of $\Ga$ in $Y_H$. 
To do so, we first prove the following:
\begin{lemma} All the stabilizers of the action of $\Ga$ on $Y$ are finite. 
\end{lemma}
\begin{proof}
By definition $Stab_{\Ga}((m,T,A)) = \{g\in \Ga| g(m,T,A) = (m,T,A)\}$. Using the discussion in the proof of the previous lemma, this subgroup of $\Ga$ can also be understood as the automorphism group of the cocycle deformation defined by $(m,T,A)$.
But we have seen, in Lemma \ref{lemma:FiniteAutomorphisms} that all such automorphism groups are finite. This implies that the dimensions of all orbits is exactly the dimension of $\Ga$ (which is $n^2$).
This also implies that all the orbits are closed, since if $\Ow\subseteq Y$ is any orbit, then $\overline{\Ow}$ is the union of $\Ow$ with orbits of smaller dimensions. 
Since there are no orbits of smaller dimensions, we deduce that $\overline{\Ow}=\Ow$.
\end{proof}
Theorem \ref{Thm:GIT} gives us then an isomorphism $K[Y_H]^{\Ga}\cong K[Y_H/\Ga]$. We will henceforth write $X_H=Y_H/\Ga$, and consider this as an affine variety.
Our next goal is therefore to describe $K[Y_H]^{\Ga}=K[X_H]$.

\begin{remark}
Since we know that cocycle deformations are necessarily of the form $\alH$, it is also possible to describe the space $X_H$ in a different way:
We can define \begin{equation}Z = \{\alpha\in (H\ot H)^*|\alpha\text{ is a 2-cocycle}\},\end{equation} and define an action of the algebraic group $(H^*)^{\times}$ by \begin{equation}\nu\cdot\alpha(x,y) = \nu(x_1)\nu(x_2)\alpha(x_2,y_2)\nu^{-1}(x_3y_3)\end{equation} where $\nu\in (H^*)^{\times}$.
The set of orbits $Z/(H^*)^{\times}$ will then be in one to one correspondence with the points of $X_H$. However, the group $(H^*)^{\times}$ might be more complicated than the group $GL(W)$, and in case $H$ is not semisimple it
will not even be reductive. For that reason we will concentrate on forming the quotient $Y/\Ga$ and not $Z/(H^*)^{\times}$. In the case of group algebras we will reduce the group $GL(W)$ to $(H^*)^{\times}$, see Section \ref{sec:Examples}.\end{remark}
\end{section}
 
\begin{section}{The algebra of invariants $K[X_H]$}\label{sec:Invariants}
In this section we will use invariant theory in order to give a full description of the ring of invariant functions on the variety $Y_H$ described in Section \ref{sec:Variety}.
We write here $\{h_1,\ldots h_n\}$ for a basis of $H$, and denote the dual basis of $H^*$ by $\{f_1,\ldots, f_n\}$.
Writing $H=\oplus_i Kh_i$ enables us to write the affine space $\A_H$ from Section \ref{sec:Variety} as 
\begin{equation}\A_H=W^{1,2}\bigoplus \oplus_{i=1}^n W^{2,1}\bigoplus \oplus_{i=1}^n W^{1,1}\end{equation}
where $W^{p,q}= W^{\ot p}\ot (W^*)^{\ot q}$ for $p,q\in\N$. In the sequel we will use freely the identification $W^{p,q}\cong \Hom_K(W^{\ot q},W^{\ot p})$. In particular $W^{p,p}\cong \End(W^{\ot p})$.
The first direct summand in the decomposition of $\A_H$ corresponds to the multiplication $m$ on $W$, the $n$ middle factors correspond to the operators $T_{h_i}$ given by $T_{h_i}(w) = T(w\ot h_i)$, 
and the last $n$ direct summands correspond to the operators $A_{f_i}$, given by $w\mapsto w_1f_i(w_2)$. 

We next rewrite the space $\A_H$ as $\A_H=\oplus_{i=1}^{2n+1} W^{p_i,q_i}$, where $(p_i,q_i)\in \{(1,1),(1,2),(2,1)\}$.
We then write an element $v=(m,T,A)\in \A_H$ as $v=(x_1,\ldots x_{2n+1})$ where we understand the elements $x_i\in W^{p_i,q_i}$ to be in one of the different direct summands of $\A_H$.
We describe the algebras $T(\A_H^*)$, $K[\A_H]$, and $K[Y_H]$, and the corresponding algebras of invariants.

\subsection{The algebras $T(\A^*_H)$ and $T(\A^*_H)^{\Ga}$}
The algebra $T(\A_H^*)$ has a direct sum decomposition \begin{equation}T(\A_H^*)=\oplus_{m=0}^{\infty} (\A_H^*)^{\ot m}.\end{equation}
The direct sum decomposition of $\A_H$ enables us to write a direct sum decomposition of $(\A_H^*)^{\ot m}$ as 
$$\bigoplus_{i_1,\ldots i_m} (W^{p_{i_1},q_{i_1}})^*\ot\cdots\ot (W^{p_{i_m},q_{i_m}})^*\cong$$
\begin{equation} \bigoplus_{i_1,\ldots i_m} (W^{p_{i_1}+\ldots +p_{i_m},q_{i_1}+\ldots + q_{i_m}})^*.\end{equation} 
This direct sum decomposition is stable under the action of $\Ga$. 
Thus, in order to calculate the subalgebra of invariants in $T(\A_H^*)$ we need to calculate the
invariant subspace of $(W^{p,q})^*$ for $p,q\in\N$. 
To describe these subspaces, we need some notations. 
For $p\in \N$ we define \begin{equation}\phi_p:KS_p\to \End(W^{\ot p})\cong W^{p,p},\end{equation}
to be the ring homomorphism which sends a permutation $\sigma\in S_p$ to the linear operator $L_{\sigma}:W^{\ot p}\to W^{\ot p}$ given by 
\begin{equation}L_{\sigma}(w_1,\ldots w_p) = w_{\sigma^{-1}(1)}\ot\cdots\ot w_{\sigma^{-1}(p)}.\end{equation}
In case $p>n$ we write \begin{equation}A_{n+1} = \frac{1}{(n+1)!}\sum_{\sigma\in S_{n+1}}(-1)^{\sigma}\sigma\end{equation} for the anti-symmetrizier in the group algebra of $S_{n+1}<S_p$.
We claim the following:
\begin{proposition}
1. The space of $\Ga$-invariants $((W^{p,q})^*)^{\Ga}$ is zero when $p\neq q$. \\
2. If $p=q$ the $\Ga$-invariants in $(W^{p,p})^*$ are spanned by the elements $T\mapsto Tr_{W^{\ot p}}(L_{\sigma}T)$ for $\sigma\in S_p$.\\
3.The kernel of the homomorphism $\phi_p$ is generated (as a two-sided ideal) by the element $A_{n+1}$ in case $p>n$, and is zero otherwise. 

\end{proposition}
\begin{proof}
The first claim is clear, since in case $p\neq q$ we can just consider the action of the nonzero scalar matrices in $\Ga$ on $W^{p,q}$ to deduce that the space of invariants is zero.
The second claim is Schur-Weyl Duality (see the discussion preceeding Theorem 1.1 in \cite{Procesi}, which is in turn based on Chapter IV of the book \cite{Weyl}), 
where we use the identification of $W^{p,p}$ with its dual, using the trace form. 
 For the third claim we proceed as follows: if $p\leq n$ then if $\{w_1,\ldots w_2,\ldots w_n\}$ is a basis of $W$ then the elements $\{\phi_p(\sigma)(w_1\ot w_2\ot\cdots\ot w_p)\}_{\sigma\in S_p}$ are linearly independent. This implies that the kernel of $\phi_p$ must be trivial.
If on the other hand $p>n$ then it holds that $A_{n+1}$ is in the kernel of $\phi_p$. To show this, we use again the basis $\{w_1\ldots w_n\}$ for $W$. It then holds that $$A_{n+1}(w_{i_1}\ot\cdots\ot w_{i_{n+1}}) =0$$ because the set $w_{i_1},\ldots, w_{i_{n+1}}$ contains a repetition for every multi-index $\{i_1,\ldots i_{n+1}\}$. This implies that $\phi_p(A_{n+1})=0$. On the other hand, it is known that the kernel of $\phi_p$ is spanned by the Young symmetrizers relative to diagrams $\lambda$ with at least $n+1$ rows (see Theorem 4.3 of \cite{Procesi}). The fact that $\lambda$ contains at least $n+1$ rows implies that the Young symmetrizer which corresponds to $\lambda$ is contained in the two-sided ideal generated by $A_{n+1}$ (see the elements $b_{\lambda}$ in Lecture 4 of \cite{Fulton-Harris}).
\end{proof}

Let us conclude this discussion by describing the algebra $T(\A_H^*)$ in terms of generators and relations:
For every multi-index $\tb{i}=(i_1,\ldots i_m)$ such that $p_{i_1}+\cdots +p_{i_m} = q_{i_1}+\cdots + q_{i_m}=r$ and a permutation $\sigma\in S_r$ we have an invariant element \begin{equation}t(\sigma,\tb{i})\in T(\A_H^*)\end{equation} (the $t$ stands here for tensor).
For an element  $\tb{x}^1\ot\cdots\ot \tb{x}^m\in (\A_H)^{\ot m}$ where $\tb{x}^i = (x^i_1,\ldots x^i_{2n+1})\in \A_H$ this invariant 
is given explicitly by \begin{equation}t(\sigma,\tb{i})(\tb{x}^1\ot\cdots\ot\tb{x}^m)=Tr_{W^{\ot r}}(L_{\sigma}(x^1_{i_1}\ot\cdots\ot x^m_{i_m})).\end{equation}
For a fixed multi-index $\tb{i}$, if $r>n$ these invariants will satisfy linear relations among them arising from the image of $A_{n+1}$ under $\phi_r$.
In other words, for every two permutations $\sigma,\tau\in S_r$ we will get the linear relation 
\begin{equation}\label{eq:Alternation}\frac{1}{(n+1)!}\sum_{\nu\in S_{n+1}}(-1)^{\nu}t(\sigma\nu\tau,\tb{i})=0.\end{equation}
The multiplication of two such $t$-invariants will again be a $t$-invariant and we have the formula
\begin{equation}\label{eq:Multiplication}t(\sigma,\tb{i})\cdot t(\tau,\tb{j}) = t((\sigma,\tau),\tb{i}\cdot\tb{j}).\end{equation}
Here $\sigma \in S_r$ where $r=p_{i_1}+\cdots p_{i_m}= q_{i_1}+\cdots q_{i_m}$ and $\tau\in S_l$ where $l=p_{j_1}+\cdots p_{j_{m'}} = q_{j_1}+\cdots q_{j_{m'}}$,
$(\sigma,\tau)\in S_{r+l}$ is the permutation which corresponds to the canonical embedding of $S_r\times S_l$ in $S_{r+l}$, and $\tb{i}\cdot\tb{j}$ is the concatenation of $\tb{i}$ and $\tb{j}$.
\begin{remark}
Notice that we have used here the fact that the canonical isomorphism $(W^{p,q})^*\cong W^{q,p}$ given by the pairing $(v_1,v_2)\mapsto Tr(v_1v_2)$ is also a $\Ga$-isomorphism.
\end{remark}
Let us conclude the above discussoin in the following Lemma
\begin{lemma}
The algebra $T(\A^*_H)^{\Ga}$ is generated by the elements $t(\sigma,\tb{i})$ subject to the relations \ref{eq:Alternation} and \ref{eq:Multiplication}.
\end{lemma}

\subsection{The algebras $K[\A_H]$ and $K[\A_H]^{\Ga}$}
The algebra $K[\A_H]$ is just the symmetric algebra $S[\A_H^*]$ since $\A_H$ is an affine space. This algebra can be described as a quotient $\pi_1:T(\A_H^*)\to S[\A_H^*]$. 	
Since the group $\Ga$ is reductive we get a surjective algebra homomorphism which we denote by the same letter $\pi_1:T(\A_H^*)^{\Ga}\to S[\A_H^*]^{\Ga}$.
The image of the element $t(\sigma,\tb{i})\in T(\A_H^*)^{\Ga}$ under $\pi_1$ is a polynomial function $p(\sigma,\tb{i})$ on $\A_H$. Since this polynomial is the image of $t(\sigma,\tb{i})$ it is given by the formula
$$p(\sigma,\tb{i})(x_1,\ldots, x_{2n+1})= t(\sigma,\tb{i})(\tb{x}\ot\cdots\ot\tb{x})=$$\begin{equation} Tr(L_{\sigma}(x_{i_1}\ot\cdots\ot x_{i_m}))\end{equation}
where $\tb{x} = (x_1,\ldots, x_{2n+1})$. 
The algebras $T(\A_H^*)$ and $S[\A_H^*]$ are both graded by $\N$. Since $S[\A_H^*]$ is the quotient of $T(\A_H^*)$ by the ideal generated by all the elements of the form $xy-yx$ where $x,y\in \A_H^*$, it holds that $$S[\A_H^*]_m = (T(\A_H^*)_m)_{S_m},$$ the coinvariants with respect to the natural action of the symmetric group $S_m$ on $(\A_H^*)^{\ot m}$. Since the actions of $\Ga$ and of $S_m$ on $(\A_H^*)^{\ot m}$ commute, it follows that 
$$S[\A_H^*]_m^{\Ga} =  ((T(\A_H^*)_m)_{S_m})^{\Ga} = (T(\A_H^*)_m^{\Ga})_{S_m}.$$ 
This means that in order to describe the algebra $S[\A_H^*]^{\Ga}$ by generators and relations we need to understand the action of the symmetric group $S_m$ on the generators $t(\sigma,\tb{i})$ of $T(\A_H^*)^{\Ga}$.

In order to describe this action, we begin with the following combinatorial definition.
\begin{definition}
Let $\tau\in S_m$, and let $c_1,\ldots ,c_m$ be a sequence of positive integers such that $\sum c_i=r$.
Write $$\{1,\ldots,r\} = I_1\sqcup I_2\sqcup\cdots\sqcup I_m = J_1\sqcup J_2\sqcup \cdots \sqcup J_m$$ where $I_1=\{1,\ldots c_1\}, I_2 =\{c_1+1,\ldots c_1+c_2\}$ and so on, and $J_1 = \{1,2,\ldots c_{\tau(1)}\}, J_2=\{c_{\tau(1)}+1,\ldots c_{\tau(1)}+c_{\tau(2)}\}$ and so on.
Then we define a permutation $\tau_{(c_i)}\in S_r$ to be the unique permutation which maps $J_i$ onto $I_{\tau(i)}$ in a monotonuous way.\end{definition}
We next prove the following lemma:
\begin{lemma}
Let $\tau\in S_m$. Assume that $a_1,\ldots a_m,b_1,\ldots b_m$ are integers such that $\sum_i a_i = \sum_i b_i = r$.
Write $\tau_1 = (\tau_{(a_i)})^{-1}$ and $\tau_2 = \tau_{(b_i)}$. Then it holds that 
$$L_{\tau_1}(y_1\ot y_2\ot\cdots\ot y_m)L_{\tau_2} = y_{\tau(1)}\ot\cdots\ot y_{\tau(m)}$$
For every $y_1\in W^{a_1,b_1},\ldots,y_m\in W^{a_m,b_m}$.
\end{lemma} 
\begin{proof}
The easiest way to see this is to think of the tensors as ``bits'', and of $y_{i}$ as a blackbox with $b_i$ ``input bits'' and $a_i$ ``output bits''. 
In this terminology, we see that the right hand side takes the first $b_{\tau(1)}$ bits via $y_{\tau(1)}$ and outputs $a_{\tau(1)}$ bits, while the left hand side takes the first $b_{\tau(1)}$ bits first to the input bits of $y_{\tau(1)}$, then applies $y_{\tau(1)}$ and apply another permutation to bring the output back to the first $a_{\tau(1)}$ bits, and so the action on the first $b_{\tau(1)}$ bits is the same for both sides of the equation. A similar phenomenon happens to the next $b_{\tau(2)}$, then to the next $b_{\tau(3)}$ bits and so on.
\end{proof}
Using the last lemma and its terminology, we can now calculate explicitly the action of $\tau\in S_m$ on $t(\sigma,\tb{i})$.
We have:
$$(\tau\cdot t(\sigma,\tb{i}))(\tb{x}^1\ot\cdots\ot\tb{x}^m) = t(\sigma,\tb{i})(\tau^{-1}\cdot(\tb{x}^1\ot\cdots\ot\tb{x}^m)) = $$
$$t(\sigma,\tb{i})(\tb{x}^{\tau(1)}\ot\cdots\ot\tb{x}^{\tau(m)}) = Tr_{W^{\ot r}}(L_{\sigma}x^{\tau(1)}_{i_1}\ot\cdots\ot x^{\tau(m)}_{i_m}) = $$
$$Tr_{W^{\ot r}}(L_{\sigma}L_{\tau_1}x^1_{i_{\tau^{-1}(1)}}\ot\cdots\ot x^m_{i_{\tau^{-1}(m)}}L_{\tau_2}) = $$
$$Tr_{W^{\ot r}}(L_{\tau_2}L_{\sigma}L_{\tau_1}x^1_{i_{\tau^{-1}(1)}}\ot\cdots\ot x^m_{i_{\tau^{-1}(m)}}) = $$
$$Tr_{W^{\ot r}}(L_{\tau_2\sigma\tau_1}x^1_{i_{\tau^{-1}(1)}}\ot\cdots\ot x^m_{i_{\tau^{-1}(m)}})=$$
$$t(\tau_2\sigma\tau_1,\tau\cdot\tb{i})(\tb{x}^1\ot\cdots\ot \tb{x}^m)$$
From which we conclude that 
$$\tau\cdot t(\sigma,\tb{i}) = t(\tau_2\sigma\tau_1,\tau\cdot\tb{i})$$ 
This translates to the relation
\begin{equation}\label{eq:Commutativity}p(\tau_2^{-1}\sigma\tau_1,\tb{i}) = p(\sigma,\tau(\tb{i})).\end{equation}
We conclude this discussion in the following lemma:
\begin{lemma}\label{lemma:palgebra}
The algebra $K[\A_H]^{\Ga}$ is generated by the elements $p(\sigma,\tb{i})$ subject to the relations \ref{eq:Alternation}, \ref{eq:Multiplication}, and \ref{eq:Commutativity}
\end{lemma}

\subsection{The algebras $K[Y_H]$ and $K[X_H] = K[Y_H]^{\Ga}$}
The algebra $K[Y_H]$ is isomorphic with $K[\A_H]/I$, where $I$ is the ideal of relations arising from the axioms of a cocycle deformation:
\begin{enumerate} 
\item $A_fA_g = A_{fg}$ for $f,g\in H^*$.
\item $m(1\ot m) = m(m\ot 1)$.
\item $A_fm = m(A_{f_1}\ot A_{f_2})$ for $f\in H^*$ and 
\item $m(1\ot A_f)T_h = f(h)Id_W$ for every $f\in H^*$ and $h\in H$.
\end{enumerate}
The reductivity of $\Ga$ implies that 
\begin{equation}\label{eq:quotient}
K[X_H]^{\Ga} = (K[Y_H]/I)^{\Ga} \cong K[Y_H]^{\Ga}/I^{\Ga}.
\end{equation}
All the above equations are equivalent to the vanishing of some polynomials in $K[\A_H]$, which in turn generate the ideal $I$. We will write down these polynomials and write a spanning set for $I^{\Ga}$ explicitly.

We start with the first axiom. 
We fix a basis $\{w_i\}$ for $W$, and we write the entries of $A_f$ by $a^f_{i,j}$ with respect to this basis. 
Thus \begin{equation}A_f(w_j)=\sum a^f_{i,j}w_i.\end{equation} Notice that $a^{f+\mu g}_{i,j} = a^f_{i,j} +  \mu a^g_{i,j}$ holds for every $f,g\in H^*$ and $\mu\in K$, and it is therefore enough to consider the elements $a^{f_k}_{i,j}$
where $\{f_k\}$ is the basis of $H^*$ described before.
The first axiom translates to the set of polynomial equations
\begin{equation}\label{eq:axiom1}\sum_k a^f_{i,k}a^g_{k,j} = a^{fg}_{i,j}\text{ for every } i,j=1,\ldots n.\end{equation}
Using the non-degeneracy of the trace form, we can write these equations as 
\begin{equation}Tr(LA_{fg})- Tr(LA_fA_g)=0\text{ for every }L\in \End_K(W).\end{equation}
The elements of the ideal generated by these polynomials will then be of the form 
\begin{equation}\label{eq:axiom2}Tr_{W^{\ot q}}(L(x_{i_1}\ot\cdots\ot x_{i_m}\ot A_fA_g))-Tr_{W^{\ot q}}(L(x_{i_1}\ot\cdots\ot x_{i_m}\ot A_{fg}))\end{equation}
for $L\in \Hom(W^p,W^q)$, where $p=\sum_j p_{i_j} + 1$ and $q=\sum_j q_{i_j} + 1$. 
It is clear how to describe the second part of \ref{eq:axiom2} (that is- the one with $A_{fg}$) in terms of the generators of $K[\A_H]$ we have described in the previous subsection. 
Notice that in case $fg$ is not itself a basis element of $H^*$ we will need to expand it as a linear combination of basis elements.
For the first part we proceed as follows: We notice that \begin{equation}A_fA_g = ev^{2,2}L_{(12)}(A_g\ot A_f),\end{equation} where $ev^{2,2}(w_1\ot w_2\ot f_1\ot f_2) = f_2(w_2)w_1\ot f_1 $ 
(that is: we evaluate the second tensor copies of $W$ and of $W^*$). This implies that 
$$Tr_{W^{\ot q+1}}(L(x_{i_1}\ot\cdots\ot x_{i_m}\ot A_fA_g)) = $$ 
\begin{equation}Tr_{W^{\ot q}}(((L\ot Id_W)L_{(q,q+1)})(x_{i_1}\ot\cdots\ot x_{i_m}\ot A_g\ot A_f))\end{equation} and the polynomial in Equation \ref{eq:axiom2} can thus be written as 
$$Tr(((L\ot Id_W)L_{(q,q+1)})(x_{i_1}\ot\cdots\ot x_{i_m}\ot A_g\ot A_f))-$$ \begin{equation}Tr(L(x_{i_1}\ot\cdots\ot x_{i_m}\ot A_{fg})).\end{equation}

We write the last polynomial as $p_{L,A_fA_g-A_{fg},\tb{i}}$. We write $I^{A_fA_g-A_{fg},\tb{i}}$ for the subspace of $I$ spanned by all the polynomials $p_{L,A_fA_g-A_{fg},\tb{i}}$. 

The rest of the axioms can also be written as the vanishing of some linear map, whose entries are polynomials in the structure constants (i.e. the entries of $A_f$, $m$ and $T$ with respect to some basis of $W$). 
For each such linear map $Q$ we define similarly the polynomials $p_{L,Q,\tb{i}}$ and the subspace $I^{Q,\tb{i}}$.
It follows that $I=\sum_{Q,\tb{i}}I^{Q,\tb{i}}$. We thus have a surjective map $\oplus_{Q,\tb{i}}I^{Q,\tb{i}}\to I$ which gives us a surjective map on the invariant subspaces since $\Ga$ is reductive:
$$\oplus_{Q,\tb{i}}(I^{Q,\tb{i}})^{\Ga}\to I^{\Ga}.$$ In order to describe the ideal $I^{\Ga}$ of $K[\A_H]^{\Ga}$ it will thus be enough to describe the subspaces $(I^{Q,\tb{i}})^{\Ga}$.
As can easily be seen, for every $Q$ the action of ${\Ga}$ on $I^{Q,\tb{i}}$ is given by $g\cdot p_{L,Q,\tb{i}} = p_{g(L),Q,\tb{i}}$, and so the invariants will again arise from invariants of the action of $\Ga$ on the spaces $W^{p,q}$,
which were already discussed before.
Using the description of the polynomials $p(\sigma,\tb{i})$ we get the following description of $I^{\Ga}$:
\begin{proposition}
For a multi-index $\emph{\tb{i}}=(i_1,\ldots, i_m)$ write $p=\sum_j p_{i_j}$ and $q=\sum_j q_{i_j}$. 
The ideal $I^{\Ga}$ is spanned by the following relations: \emph{
\begin{equation}\label{eq:Axioms}p(\sigma,\tb{i}\circ A_{fg}) -p(\sigma(p,p+1),\tb{i}\circ A_g\circ A_f)\text{ for } p=q, \sigma\in S_{p+1},\end{equation}
$$p((p+1,p+2)\sigma(p+1,p+2),\tb{i}\circ m\circ m) - $$ $$p((p,p+2)\sigma, \tb{i}\circ m\circ m) \text{ for } p=q+2, \sigma\in S_{p+1}$$
$$p(\sigma(p+1,p+2,p+3),\tb{i}\circ A_{f_1}\circ A_{f_2}\circ m) -$$ $$ p(\sigma(p+1,p+2),\tb{i}\circ m\circ A_f) \text{ for }p=q+1, \sigma\in S_{p+1},$$
$$p(\sigma(p+1,p+3,p+4),\tb{i}\circ T_h\circ A_f\circ m) -$$ $$ f(h)p(\sigma,\tb{i}\circ Id_W)\text{ for } p=q, \sigma\in S_{p+1}$$}

where \emph{$$p(\sigma,\tb{i}\circ Id_W) = 
\begin{cases}
dim(W) p(\sigma,\tb{i}) \text{ if } \sigma(p+1)=p+1 \\
p(\sigma (r,p+1),\tb{i}) \text{ if } \sigma(r)=p+1
\end{cases}
$$}
\end{proposition}
Intuitively, we can understand the relations appearing in the proposition in the following way:
All the invariants $p(\sigma,\tb{i})$ can be understood as traces of maps formed from the maps $T,A,m$ by composition and by permuting the tensor factors. 
The relations appearing in the proposition say that if two linear maps are equal, then switching between them inside the relevant invariants will not change the invariants.
We conclude with the following description of $K[X_H]$.
\begin{theorem}\label{thm:main2}
 The algebra $K[X_H]$ is generated by the elements \emph{$p(\sigma,\tb{i})$}. The ideal of relations among these elements is generated by the polynomials which appear in Equations \ref{eq:Alternation}, \ref{eq:Multiplication} , \ref{eq:Commutativity}, and \ref{eq:Axioms}. 
\end{theorem}

\begin{proof}
This follows from Equation \ref{eq:quotient} and Lemma \ref{lemma:palgebra}. together with the discussion above about the set of generators for $I^{\Ga}$.
\end{proof}

\begin{remark}
Notice that it is possible that the ideal $I$ defined above might not be a radical ideal. This makes little difference for us, since in general if $I$ is an ideal in a commutative ring $R$ upon which a group $\Ga$ acts, 
then the $R^{\Ga}$ ideals $rad(I)^{\Ga}$ and $rad(I^{\Ga})$ can easily be shown to be equal.
\end{remark}

We next describe a simplified form for the invariants $p(\sigma,\tb{i})$. 
For this, we define, for $h(1),\ldots h(l)\in H$, 
$$T(h(1),\ldots, h(l)): W\to W^{\ot (l+1)}\text{ as the composition}$$ \begin{equation} (Id_W^{\ot l-1}\ot T_{h(l)})\cdots(Id_W\ot T_{h(2)})T_{h(1)}.\end{equation}
In a similar way we define \begin{equation}m^l:W^{\ot l+1}\to W, \quad m^l=m\cdot (m\ot Id_W)\cdots(m\ot Id_W^{\ot l-1}).\end{equation}
We consider some identities for the maps $T_h$, $m$, and $A_f$. Notice that the first one is one of the axioms for a cocycle deformation which we repeat here since we will use it directly in the next proposition. 
All the rest can be proven using the isomorphism $W\cong \allH$ for some 2-cocycle $\alpha$.
\begin{lemma}\label{lemma:Identities} The following identities hold:\\
1. $A_f m = m(A_{f_1}\ot A_{f_2})$ \\
2. $(1\ot A_f)T_h = T_{h_1f(h_2)}$ and \\
3. $(A_f\ot 1)T_h = T_{f_2(S(h_1))h_2}A_{f_1}$\\ 
4. $(T_h\ot 1)T_g= (1\ot T_{g_2})T_{hg_1}$.
\end{lemma}

\begin{proof}
	The first identity is a reformulation of the fact that $W$ is a comodule algebra. In other words, that the multiplication map commutes with the action of $H^*$ given by $f\mapsto A_f$. For the rest of the identities, we use the formula for $T$ from the proof of Lemma \ref{lemma:Minv}. We have $$(1\ot A_f)T_h(x) = (1\ot A_f)(x\cdota \tiS(h_1)\ot h_2) = x\cdota\tiS(h_2)\ot h_2f(h_3) = T_{h_1f(h_2)}(x)$$ as desired. The proof of identities 3 and 4 are similar.
	\end{proof}
	 
We now claim the following:
\begin{proposition}
 Each invariant \emph{$p(\sigma,\tb{i})$} can be written as a sum of invariants of the form $Tr_W(m^lL_{\sigma'}T(h(1),\ldots,h(l))A_f)$. 
\end{proposition}
\begin{proof}
We have $p(\sigma,\tb{i})= Tr_{W^{\ot p}}(L_{\sigma}x_{i_1}\ot\cdots\ot x_{i_m})$. We use the fact that $M:W\ot W\to W\ot H$ is invertible with inverse $T:W\ot H\to W\ot W$.
We conjugate the map which appears in the definition of $p(\sigma,\tb{i})$ with the map \begin{equation}\ti{M}:=(M\ot Id_H^{\ot p-2})\ot\cdots\ot (Id_W^{\ot p-3}\ot M\ot Id_H)(Id_W^{\ot p-2}\ot M)\end{equation} and get 
\begin{equation}p(\sigma,\tb{i}) = Tr_{W\ot (H^{\ot p-1})}(\ti{M} L_{\sigma} x_{i_1}\ot\cdots\ot x_{i_m} \ti{T})\end{equation}
where $\ti{T} = \ti{M}^{-1} = (Id_W^{\ot p-2}\ot T)\cdots (T\ot Id_H^{\ot p-2})$.
By taking the basis $\{h_i\}$ of $H$ and the dual basis $\{f_i\}$ of $H^*$, we can rewrite the last trace as the sum of $n^{p-1}$ traces of maps from $W$ to $W$.
Indeed, if $R:W\ot H^{\ot p-1}\to W\ot H^{\ot p-1}$ is any linear map then $$Tr_{W\ot H^{\ot p-1}}(R) = $$ 
\begin{equation} \sum_{j_1,\ldots j_{p-1}}Tr_W(Id_W\ot f_{j_1}\ot\cdots\ot f_{j_{p-1}})R(-\ot (h_{j_1}\ot\cdots\ot  h_{j_{p-1}})).\end{equation}
Using the first identity in the previous lemma and the definition of $M$ we see that the invariant $p(\sigma,\tb{i})$ is the sum of traces of maps of the form \begin{equation}m^{p-1}(A_{f'_1}\ot\cdots\ot A_{f'_p}) L_{\sigma}(x_{i_1}\ot\cdots\ot x_{i_m})T(h_{j_1},\ldots h_{j_{p-1}}).\end{equation}
It will thus be enough to prove the statement of the proposition for maps of the above form. We first use the fact that 
$A_{f'_1}\ot\cdots\ot A_{f'_p}L_{\sigma} = L_{\sigma}A_{f'_{\sigma(1)}}\ot\cdots\ot A_{f'_{\sigma(p)}}$ to re-write this map as 
\begin{equation}m^{p-1}L_{\sigma}(A_{f''_1}\ot\cdots\ot A_{f''_p})(x_{i_1}\ot\cdots\ot x_{i_m})T(h_{j_1},\ldots h_{j_{p-1}}).\end{equation}
The tensors $x_{i_j}$ are of the form $A_f$ or $T_h$ or $m$. 
We can then use Identities 1-3 from Lemma \ref{lemma:Identities} to ``push to the right'' all the appearances of $A_f$ (including those in $x_{i_j}$) 
to the beginning of the linear map, to get a map of the form 
\begin{equation}m^{p-1}L_{\sigma}x_{i'_1}\ot\cdots\ot x_{i'_{m'}}T(h_{j_1},\ldots, h_{j_{p-1}})A_f\end{equation}
where $x_{i'_j}$ is either $T_h$ or $m$.
This map can be written as \begin{equation}m^{p-1}L_{\sigma}S_1S_2 T(h_{j_1},\ldots h_{j_{p-1}})A_f\end{equation}
where $S_1$ is a tensor product of copies of $m$ and of $Id_W$, and $S_2$ is a tensor product of copies of $T_{h_i}$ and $Id_W$. 
We can write $L_{\sigma}S_1 = S_3L_{\sigma'}$ where $S_3$ is again a tensor product of copies of $m$ and of $Id_W$ and $\sigma'\in S_{l+1}$ is a permutation defined from $\sigma$.
We can also use Identity 4 from Lemma \ref{lemma:Identities} repeatedly to write the composition $S_2T(h_{j_1},\ldots, h_{j_{p-1}})$ as a sum of maps of the form $T(h'_1,\ldots h'_{l})$. 
Using the associativity of $m$, we arrive at a map of the form $m^{l}L_{\sigma'}T(h'_1,\ldots, h'_{l})A_f$, which is what we wanted to prove.
\end{proof}
We write $$c(l,\sigma,f,h(1),\ldots h(l)) = Tr_W(m^lL_{\sigma}T(h(1),\ldots,h(l))A_f).$$
The above proposition implies immediately the following theorem, which together with Theorem \ref{thm:main2} above finishes the proof of Theorem \ref{thm:main}.
\begin{theorem}
The scalar invariants $c(l,\sigma,f,h(1),\ldots h(l))$ form a complete set of invariants for $W$. We call these invariants the basic invariants of $W$.
\end{theorem}

We can, in principal, translate all the relations which the invariants $p(\sigma,\tb{i})$ satisfy to relations for the basic invariants.
In both cases the set of invariants and the set of relations is quite big. We will nevertheless be able to calculate some invariants explicitly, using Proposition \ref{prop:Reduction}.

As a first application of the basic invariants we study the relation of the invariants with Galois theory. 
For the next proposition, assume that $k\subseteq K$ is a subfield of $K$ such that $K/k$ is a Galois extension with Galois group $G$. 
Assume also that $H$ is already defined over $k$. In other words, assume that there is a Hopf algebra $H_k$ over $k$ such that $H\cong H_k\ot_k K$. 
This is true for example in case $K=\wb{\Q}$, $H=KZ$ or $K[Z]$ for some finite group $Z$, and $k=\Q$. In this case we will call a basic invariant $c(l,\sigma,f,h(1),\ldots,h(l))$ \textit{$k$-rational} if $f\in H_k^*$ and $h(1),\ldots, h(l)\in H_k$
(we can think of $H_k$ and $H_k^*$ as subsets of $H$ and $H^*$ respectively). 
 We claim the following:
 \begin{proposition}\label{prop:Galois}
Let $W$ be a cocycle deformation of $H$, defined over $K$.
  Let $L=k(c(l,\sigma,f,h(1),\ldots h(l)))\subseteq K$ be the subfield of $K$ generated by the $k$-rational basic invariants of $W$.
  Then for $\gamma\in G$ it holds that $^{\gamma}W\cong W$ if and only if $\gamma$ fixes the subfield $L$ of $K$ pointwise.
 \end{proposition}
\begin{proof}
Recall first that $^{\gamma}W$ is the vector space $W$ twisted by the action of $\gamma$. Since the Hopf algebra $H$ is already defined over $k=K^{G}$, $^{\gamma}W$ is again a cocycle deformation of $H$.
The structure constants for the multiplication and the coaction are given by applying $\ga$ to the structure constants of $W$. The fact that the Hopf algebra $H$ is already defined over $k$ is crucial here, since otherwise we would get that $^{\ga}W$ is a comodule algebra over $^{\ga}H$. For more on Galois twisting of algebraic structures, see Section 8 of \cite{Meir1}.
It is easy to show that since $f\in H_k^*$ and $h(i)\in H_k$ it holds that

$$c_{^{\gamma}W}(l,\sigma,f,h(1),\ldots h(l)) = Tr_{^{\gamma}W}(m^lL_{\sigma}T(h(1),\ldots,h(l))A_f).$$
Let us fix a basis $\{w_1,\ldots w_n\}$ for $W$. We can write all the linear maps $T(h(i))$, $A_f$ and $m$ with respect to that basis, using structure constants (this is the basic idea which enables us to consider $W$ as a point in an affine space in the first place).
For the sake of simplicity, let us write $x_1,\ldots, x_a$ for the set of structure constants which appear in the map $m^lL_{\sigma}T(h(1),\ldots,h(l))A_f$. Then the trace of this map will be a polynomial $p(x_1,\ldots, x_a)$ with rational coefficients (this follows easily by considering the fact that composition of linear maps is given by polynomial with rational coefficients in the structure constants).
The structure constants for $^{\gamma}W$ will then be $\gamma(x_1),\ldots, \gamma(x_a)$.
This implies that the trace for $m^lL_{\sigma}T(h(1),\ldots,h(l))A_f$ considred as a map $^{\gamma}W\to ^{\gamma}W$ will be 
$p(\gamma(x_1),\ldots, \gamma(x_a)) = \gamma(p(x_1,\ldots x_a))$, because $p$ is a polynomial with rational coefficients.
From this discussion, we get the equation
 $$\gamma(c_W(l,\sigma,f,h(1),\ldots h(l))) = c_{^{\gamma}W}(l,\sigma,f,h(1),\ldots h(l)).$$
But the cocycle deformations $W$ and $^{\gamma}W$ are isomorphic if and only if they have the same basic invariants. 
Since the basic invariants with $f\in H_k^*$ and $h(i)\in H_k$ determine all the other basic invariants, this implies the proposition.
\end{proof}
\end{section}

\begin{section}{Examples}\label{sec:Examples}
We will give here examples for the invariants one gets for group algebras, dual group algebras, and several non-semisimple Hopf algebras.
\subsection{Group algebras}
In this case the Hopf algebra is a group algebra $H=KG$ and the cocycle deformation algebra is of the form $W=K^{\alpha}G$, where $[\alpha]\in H^2(G,K^{\times})$ is a cohomology class in the usual group cohomology sense.
The algebra $W$ has a basis given by $\{U_g\}$ for $g\in G$, and the multiplication is given by $U_gU_h = \alpha(g,h)U_{gh}$.
The map $T$ is given by $T(U_g\ot h) = U_gU_h^{-1}\ot U_h$. The map $A_{e_g}:W\to W$ sends $U_h$ to $\delta_{g,h}U_h$. Consider the invariant 
\begin{equation}\small{c(l,\sigma,f,e_g,g(1),\ldots g(l)) = Tr_W(m^{l}L_{\sigma}T(g(1),g(2),\ldots g(l))A_{e_g}).}\end{equation}
The map $m^{l}L_{\sigma}T(g(1),g(2),\ldots g(l))A_{e_g}$ will send $U_h$ to zero for $h\neq g$, and $U_g$ to $U_{h_1}^{\ep_1}U_{h_2}^{\ep_2}\cdots U_{h_{2l+1}}^{\ep_{2l+1}}$, 
where $\ep_i=\pm 1$, the list $\{h_1,\ldots h_{2l+1}\}$ contains every $g_i$ once with a positive and once with a negative power, and $g$ once with a positive power.
If $g\neq h_1^{\ep_1}\cdots h_{2l+1}^{\ep_{2l+1}}$ this map is nilpotent, and will therefore have trace zero. In the other case, where 
$g= h_1^{\ep_1}\cdots h_{2l+1}^{\ep_{2l+1}}$,
this map will have $U_g$ as an eigenvector with nonzero eigenvalue, and its trace will be $U_{h_1}^{\ep_1}U_{h_2}^{\ep_2}\cdots U_{h_{2l+1}}^{\ep_{2l+1}}U_g^{-1}\in K$. 
We have an interpretation of this invariant in terms of the Universal Coefficients Theorem: Indeed, the Universal Coefficient Theorem gives us a homomorphism 
\begin{equation}\Phi:H^2(G,K^{\times})\cong \Hom(H_2(G,\Z),K^{\times}).\end{equation}
The second integral homology group, also known as the Schur Multiplier of $G$ can be described in the following way: if 
\begin{equation}1\to R\to F\to G\to 1\end{equation} is a free resolution of $G$, then $H_2(G,\Z) \cong ([F,F]\cap R)/[F,R]$.
If we consider the free group generated by the symbols $\{x_g\}$ and mapped to $G$ in the obvious way, and if $h_1,\ldots h_{2l+1}$ is a sequence as above, then $x_{h_1}^{\ep_1}\cdots x_{h_{2l+1}}^{\ep_{2l+1}}x_g^{-1}\in [F,F]\cap R$. 
It then holds that \begin{equation}\Phi([\alpha])(x_{h_1}^{\ep_1}\cdots x_{h_{2l+1}}^{\ep_{2l+1}}x_g^{-1})\in K\end{equation} is exactly the above invariant. Also, it is easy to see that every element $a\in [F,F]\cap R$, will give rise to a basic invariant 
(by abuse of notations, we think here of homomorphisms $H_2(G,\Z)\to K^{\times}$ as homomorphisms from $[F,F]\cap R$ which vanish on $[F,R]$).
We summarize this in the following proposition:
\begin{proposition}
In case $H=KG$ all basic invariants are either zero or the invariants of the cocycle arising from the isomorphism of the Universal Coefficients Theorem.
\end{proposition}
This gives us a description of the basic invariants, but it does not give us a description of the relations between them. To get a concrete description, which will in fact reconstruct for us the isomorphism $\Phi$, 
we will use reduction of the acting group (Proposition \ref{prop:Reduction}).

For this, we enumerate the group elements $\{1=g_1,\ldots, g_n\}$ of $G$. 
We fix a basis $\{w_1,\ldots, w_n\}$ of $W$. We consider the variety $Y$ of all cocycle deformation structures on $W$. 
We write $\Ga= GL(W)$. Let \begin{equation}N=\{\ga\in \Ga| \ga(w_i)=\lambda_i w_i\text{ for some } \lambda_i\in K^{\times}\}\cong GL_1(K)^n.\end{equation} 
We write $Y'\subseteq Y$ for the subvariety \begin{equation}Y'=\{((A_g),m,(T_{g}))| \forall i \, A_{g_i} = e_{ii}\}.\end{equation} 
We claim that the subgroup $N$ stabilizes the subvariety $Y'$, and every orbit of $\Ga$ in $Y$ intersects $Y'$ in exactly one $N$-orbit. 
Indeed, since for every cocycle deformation we will have a direct sum decomposition $W=\oplus_{g\in G} W_g$ where $W_g=span\{U_g\}$ is one dimensional, it holds that every orbit of $\Ga$ in $Y$ intersects $Y'$, 
and if two points in $Y'$ are conjugate under the action of $\ga\in \Ga$, then $\ga$ fixes all the maps $A_{g_i}=e_{ii}$ and is therefore contained in $N$. 
The group $N$ is also reductive.
We thus see that all the conditions of Proposition \ref{prop:Reduction} hold, and we can therefore reduce the structure group from $\Ga$ to $N$. 

Notice that we previously had $n^3$ entries for the maps $A_g$, and $n^3$ more entries for the multiplication. Since we passed to $Y'$ the maps $A_g$ are given to us now explicitly. 
The conditions saying that the multiplication respects the group grading will boil down to the polynomial equations $m_{i,j}^k=0$ if $g_ig_j\neq g_k$. 
This means that we have de-facto only $n^2$ entries for the multiplication. These entries will be exactly the values of the cocycle $\alpha$. 
The variety we are left with is exactly the variety of 2-cocycles on $G$, and the acting group acts by multiplying by one-cochains.

We re-write this variety in the following way: $$Y'=\{(\alpha(g,h))_{g,h\in G}|\forall g,h,k\in G\; \alpha(g,h)\alpha(gh,k) = \alpha(h,k)\alpha(g,hk),$$ \begin{equation}\alpha(g,h)\neq 0\}\end{equation}
(in fact, this is only a variety which is isomorphic with $Y'$ as an $N$-variety, but this does not make a big difference for us).
The group $N=GL_1^n$ acts on this variety by the action \begin{equation}(\lambda_g)_{g\in G}\cdot (\alpha(g,h))_{g,h\in G} = (\alpha(g,h)\lambda_g\lambda_h\lambda_{gh}^{-1})_{g,h\in G}.\end{equation}
We thus need to describe the $GL_1^n$ invariants in $K[Y']$.
We first notice that $K[Y']$ is in fact the group algebra of the abelian group $A$ generated by the elements $\alpha(g,h)$ modulo the relations arising from associativity, that is:
\begin{equation}A= \langle \alpha(g,h)\rangle / \langle \alpha(g,h)\alpha(gh,k)\alpha^{-1}(h,k)\alpha^{-1}(g,hk)\rangle\end{equation}and $K[Y'] = KA$.

For every $a\in A$ the subspace $Ka\subseteq K[Y']$ is a one dimensional representations of $GL_1^n$.
Since the character group of $GL_1^n$ is $\Z^n$, the action induces a homomorphism of abelian groups $\psi:A\to \Z^n$. 
In other words, for every $a\in A$ and $(\lambda_g)\in N$ it holds that 
\begin{equation}(\lambda_g)\cdot a = \prod_{i=1}^n \lambda_{g_i}^{\psi_i(a)}a\end{equation}
where $\psi=(\psi_1,\ldots,\psi_n)$.

The invariant subspace $K[Y']^{GL_1^n}\cong KA^{GL_1^n}$ is then spanned by $ker(\psi)\subseteq A$. We claim the following:
\begin{lemma}
We have an isomorphism $\Xi:ker(\psi)\cong ([F,F]\cap R)/[F,R]$ where $F=\langle x_g|g\in G\rangle$ is a free group and $R$ is the kernel of the canonical projection $F\to G$. 
\end{lemma}
\begin{proof}
We will construct explicitly $\Xi$ and its inverse $\Xi^{-1}$.
Assume first that $a=\prod_{i=1}^m \alpha(g_i,h_i)^{\ep_i}\in ker(\psi)$,
where as usual $\ep_i\in\{\pm 1\}$.
We consider the element $a'=\prod_{i=1}^m(x_{g_i}x_{h_i}x_{g_ih_i}^{-1})^{\ep_i}\in F$.
The element $a'$ is contained in $R$ since it is a product of elements in $R$.
The fact that $a\in ker(\psi)$ implies that the element $a'$ is contained in $[F,F]$. Indeed, the exponent sum of $x_h$ for $h\in G$ in $a'$ is equal to $b_h$, 
where $(\lambda_g)\cdot a = (\prod_g \lambda_g^{b_g})a'$. But since $a'\in ker(\psi)$ this implies that this sum is zero. It follows that the exponent sum is zero for every $h\in G$, and $a'$ is thus an element of $[F,F]\cap R$.

We would like to define now $\Xi(a) = a' + [R,F]$. 
For this, we need to check that this map is well defined. Two problems may arise here: the first one is the fact that we took the product in $a'$ in a specific order. 
The second one is that the elements $\alpha(g,h)$ do not freely generate the abelian group $A$, and they satisfy relations arising from the associativity in the algebra.

For the first problem, we just need to consider the commutation of two instances of $\alpha$ variables: 
namely elements of the form \begin{equation}\alpha(g_1,h_1)^{\ep_1}\alpha(g_2,h_2)^{\ep_2}\alpha(g_1,h_1)^{-\ep_1}\alpha(g_2,h_2)^{-\ep_2}.\end{equation} 
But this will be the same as the commutator of two elements in $R$, which is obviously in $[R,F]$.
For the second problem, notice that all the relations among the $\alpha$ variables are generated in degree zero, 
namely by elements of the form $a=\alpha(g,h)\alpha(gh,k)\alpha(g,hk)^{-1}\alpha(h,k)^{-1}$.
But a direct calculation shows that in this case $a'$ is an element in $[F,R]$, so $\Xi$ is well defined indeed.

In order to define $\Xi^{-1}$ we proceed in a similar way.
For every element $x\in [F,F]\cap R$ we write $x$ as the product $\prod_i r_{g_i,h_i}^{\ep_i}$ where $r_{g,h} = x_gx_hx_{gh}^{-1}$ (this is possible since the elements $r_{g,h}$ generate $R$ as a group).
We then define $x' = \prod_i\alpha(g_i,h_i)^{\ep_i}\in A$. 
We then use the fact that $x\in [F,F]$ to show that $x'\in ker(\psi)$, 
and we use the fact that $A$ is an abelian group with defining relations arising from the associativity of $\alpha$ in order to show that if $x\in [R,F]$ then $x'$ is trivial in $A$. 
This then enables us to define a group homomorphism $([F,F]\cap R)/ [F,R]\to ker(\psi)\subseteq A$ by $\Theta(x)=x'$ and show that $\Theta=\Xi^{-1}$. This finishes the proof of the lemma.
\end{proof}
Notice that with the proof of the above lemma we have re-established the special case of the Universal Coefficients Theorem:
indeed, we have seen that an equivalence class of a 2-cocycle is equivalent to a ring homomorphism $K[Y]^{GL_1^n}\to K$.
But by the above isomorphism, this is the same as an abelian group homomorphism $[F,F]\cap R/[F,R]\to K^{\times}$, which is the isomorphism given by the Universal Coefficients Theorem.
 
\subsection{Dual group algebras}
Consider the case where the Hopf algebra is a dual group algebra $H=K[G]$. 
We will describe the basic invariants we receive here.
In this case, cocycle deformations of $H$ will be associative algebras with a $G$-action, which are isomorphic to $KG$ as $G$-representations.
We use now the fact that cocycle deformations of $H$ are the same thing as Drinfeld twists on $H^*$. Following \cite{Movshev} and \cite{EG3} (see also \cite{GN} and Theorem 3.2 in \cite{Ostrik}) we have the following classification result:
\begin{proposition}\label{prop:dualgroup}
A cocycle deformation for $K[G]$ is given by a pair $(F,[\alpha])$  where $F$ is a subgroup of $G$ and $[\alpha]\in H^2(F,K^{\times})$ 
is a non-degenerate 2-cocycle (which means that $K^{\alpha}F$ is isomorphic with a full matrix algebra).
Two such pairs will give rise to equivalent cocycle deformations if and only if they differ by conjugation by an element of $G$.
The algebra which corresponds to $(F,[\alpha])$ is given explicitly by $\oplus_{i}e_{t_i}K^{\alpha}F$ where $\{t_i\}_i$ is a set of coset representatives of $F$ in $G$,
and the $e_{t_i}$ are pairwise orthogonal idempotents. 
The action of $G$ on this algebra is given explicitly in the following way: 
if $gt_i = t_jf$ then $$g\cdot (e_{t_i}U_{f'}) = e_{t_j}U_fU_{f'}U_f^{-1}.$$
\end{proposition}

We can, in principle, write a cocycle on $K[G]$ which arises from the cocycle on $F$. 
For this, one uses the fact that as an $F$-module under conjugation $K^{\alpha}F$ is isomorphic with $KF$ with the regular action. 

We will take a different approach here. 
We will write the linear maps $T_f$ explicitly, and we will use them to write down the invariants explicitly.
We will not describe explicitly the relations between these invariants, but we will show how they relate to the basic invariants of the cocycle $\alpha$, 
and we will say something about their arithmetic properties.

The map $M$ sends $e_{t_i}U_{f_1}\ot e_{t_j}U_{f_2}$ to $\sum_{f\in F}e_{t_i}U_{f_1}U_fU_{f_2}U_f^{-1}\ot e_{t_ift_j^{-1}}$.
Notice that the idempotent in the second tensor factor is an element of $K[G]$, and not of $W$.
We shall now use the fact that the cocycle $\alpha$ is non-degenerate in order to describe explicitly the inverse $T$ of $M$.
For this, we will use the following facts: since $K^{\alpha}F\cong M_n(K)$, this algebra has a trivial center. 
For every $1\neq f\in F$ the element $\sum_{f'\in H}U_{f'}U_fU_{f'}^{-1}$ is central and not a scalar multiple of the identity, and is therefore zero. 

For an element $f'\in F$ and two coset representatives $t_i$ and $t_j$ we consider now the element 
\begin{equation}X_{i,j,h_1,h_2}= \frac{1}{|F|}\sum_{f\in F}e_{t_i}U_{f_1}U_{f}\ot e_{t_j} U_{f_2}^{-1}U_f^{-1}U_{f_2}.\end{equation} Its image under $M$ is the element
$$\frac{1}{|F|}\sum_{f,f_3\in H} e_{t_i} U_{f_1}U_fU_{f_3}U_{f_2}^{-1}U_f^{-1}U_{f_2}U_{f_3}^{-1}\ot e_{t_if_3t_j^{-1}} = $$
\begin{equation}e_{t_i} U_{f_1}\ot e_{t_if_2t_j^{-1}}.\end{equation}
For the equality we have used the fact that the sum \begin{equation}\frac{1}{|F|}\sum_{f\in F}U_fU_{f_3}U_{f_2}^{-1}U_f^{-1}\end{equation} is zero, unless $f_3=f_2$, in which case it is 1. 
This already gives us the inverse of the map $M$: indeed, for $t_i\in T$ and $g\in G$ there is a unique index $j$ such that $f_2:=t_i^{-1}gt_j$ is in $H$ 
(this follows easily from the fact that $t_j^{-1}$ are the representatives of the right cosets of $F$ in $G$). We then have that \begin{equation}T(e_{t_i} U_{h_1}\ot e_g) = X_{i,j,f_1,t_i^{-1}gt_{j}}.\end{equation}
Writing this more explicitly, we have that 
\begin{equation}T_{e_g}(e_{t_i} U_{f_1}) = \frac{1}{|F|}\sum_{f\in F}e_{t_i}U_{f_1}U_{f}\ot e_{t_j} U_{f_2}^{-1}U_f^{-1}U_{f_2}\text{ where } f_2=t_i^{-1}gt_j.\end{equation}

We now turn to the calculation of the invariants. 
We will begin with the case $F=G$.
In this case our algebra $W$ is just $K^{\alpha}F$ and \begin{equation}T_{e_g}(U_f) = \frac{1}{|F|}\sum_{f'} U_{f}U_{f'}\ot U_g^{-1}U_{f'}^{-1}U_g.\end{equation}
The map $T(e_{g(1)},\ldots e_{g(l)})$ sends $U_f$ to 
\begin{equation}\frac{1}{|F|^{l}}\sum_{f'_1,\ldots,f'_{l}} U_fU_{f'_1}\cdots U_{f'_{l}}\ot U_{g(l)}^{-1}U_{f'_{l}}^{-1}U_{g(l)}\ot\cdots\ot U_{g(1)}^{-1}U_{f'_1}^{-1}U_{g(1)}. \end{equation}
We would like to show that the trace of the map $A_g m^{l}L_{\sigma}T(e_{g(1)},\ldots,e_{g(l)})$ is the sum of invariants of the cocycle $\alpha$.
In order to do so we first show that the projection $$E_f:K^{\alpha}F\to K^{\alpha}F$$ \begin{equation}\sum_{f'\in F} a_{f'}U_{f'}\mapsto a_{f}U_{f}\end{equation} can be expressed using the maps $T_{e_g}$. 
Indeed, consider the map $m\tau T_{e_f}$ where $\tau$ is the flip of vector spaces. This is the same as the action of $K[G]$ on $W$ described in Section \ref{sec:Prelim}. We have 
\begin{equation}m\tau T_{e_f}(U_{f''}) = \frac{1}{|F|}\sum_{f'}U_f^{-1}U_{f'}^{-1}U_fU_{f''}U_{f'}.\end{equation}
In case $ff''\neq 1$ this sum is zero. When $ff''=1$ then this sum is equal to
\begin{equation}\frac{1}{|F|}\sum_{f'}U_f^{-1}U_fU_{f''}U_{f'}^{-1}U_{f'} = U_{f''},\end{equation}
and we thus see that $m\tau T_{e_{f^{-1}}}=E_f$.
We write \begin{equation}A_gm^{l}L_{\sigma}T(e_{g(1)},\ldots e_{g(l)}) = \sum_{f_1,\ldots f_l} A_gm^{l}L_{\sigma}(E_{f_1}\ot \cdots\ot E_{f_l})T(e_{g(1)},\ldots e_{g(l)}).\end{equation}
The trace of the left hand side will thus be the sum of the traces of the maps on the right hand side. But following the calculation we have done for group algebras, 
for every $l$-tuple $f_1,\ldots f_l$ the trace of any of the maps on the right hand side is either zero or $\frac{1}{|F|^{l}}$ times one of the basic invariants of the cocycle $\alpha$. 
Moreover, it can easily be shown that all the invariants of the cocycle $\alpha$ will be received in this way. 
Lastly, notice that the fact that $E_f$ can be expressed using $m$ and $T$ implies that the summands on the  right hand side are also basic invariants. 
We summarize this discussion in the following proposition:
\begin{proposition}
In the case $F=G$ all the basic invariants are either zero or of the form $$|F|^{-r}\sum_i c_i$$ where all the $c_i$ are basic invariants of the cocycle $\alpha$.
\end{proposition}

We next consider the case where $F$ is a proper subgroup of $G$. 
We write $W=\oplus W_i$ where $W_i = e_{t_i}K^{\alpha}H$. 
We would like to calculate the trace of linear maps of the form $A_gm^{l}L_{\sigma}T(e_{g(1)},\ldots e_{g(l)})$.
We consider the action of this linear map on $W_i$. First notice that if for some $k=1,\ldots l$ the index $j$ for which $t_jF = g_kt_iF$ is different from $i$, this map will be the zero map:
this follows directly from the fact that the multiplication $e_{t_i} U_{f_1}\cdot e_{t_j} U_{f_2}$ is zero if $i\neq j$.
We thus see that this trace will be zero unless $g_k\in t_iFt_i^{-1}$ for every $k$.
Notice also that if $g\notin t_iHt_i^{-1}$ the action of $A_g$ will send $W_i$ to $W_j$ for some $j\neq i$, and the trace will be zero again. 
Let now 
\begin{equation}d_F(g,\sigma,g(1),\ldots g(l)) = \begin{cases}  
Tr_{W_1}(A_gm^{l}L_{\sigma}T(e_{g(1)},\ldots, e_{g(l)}) \text{ if }g,g_1,\ldots g_{l}\in F \\ 
0 \qquad \text{ else} \end{cases}\end{equation}
These are the invariants which we encountered in the case $F=G$. The above discussion can be summarized in the following proposition:
\begin{proposition}
The basic invariants of $W$ are of the form $$c(l,\sigma,g,e_{g(1)},\ldots, e_{g(l)}) = \sum_i d_F(t_i^{-1}gt_i,\sigma,t_i^{-1}g(1)t_i,\ldots,t_i^{-1}g(l)t_i)$$ 
\end{proposition}
As a more specific example, consider the case where $F$ is a normal subgroup of $G$. In this case, if any of the elements $g,g(1),\ldots g(l)$ is not in $F$, then the invariant is automatically zero. 
If all these elements are in $F$, then $d_F(g,\sigma,g(1),\ldots,g(l))$ is a sum of basic invariants of the cocycle $\alpha$. 
In a similar fashion, for every $i$ the scalar $d_F(t_i^{-1}gt_i,\sigma,t_i^{-1}g(1)t_i,\ldots,t_i^{-1}g(l)t_i)$ will be the sum of the same basic invariants, but for the cocycle $t_i^*(\alpha)$.

\subsubsection{Dual group algebras- a concrete example.}  Assume that the ground field $K$ is $\overline{\Q}$. 
Consider the group $F=\Z/3\times \Z/3$. Let $\{x,y\}$ be a basis for $F$, considered as a vector space over $\Z/3$. Consider the following action of $\Z/4 = \langle g \rangle$ on $F$: $g(x^iy^j) = x^{2i}y^j$,
and construct the semidirect prouct $G = (\Z/3\times\Z/3)\rtimes \Z/4$. Notice that the subgroup $2\Z/4$ of $\Z/4$ is central.

Let $\alpha:F\times F\to K^{\times}$ be the 2-cocycle defined by $(x^iy^j,x^ky^l) = \zeta^{jk}$ where $\zeta\in K$ is a third root of unity.
This cocycle is non-degenerate, and a direct calculation shows that $g^*(\alpha)$ is cohomologous to $\nu^*(\alpha)$ where $\nu\in Gal(\Q(\zeta)/\Q)$ sends $\zeta$ to $\zeta^{-1}$.
Then we see that in this case all the basic nonzero invariants will be of the form $c(g,\sigma,g(1),\ldots, g(l))= (2+2\nu)(d_F(g,\sigma,g(1),\ldots g(l)))$ for some $g,g(1),\ldots g(l)\in F$, and will therefore be \textit{rational}.

We thus have the following situation: the Hopf algebra $K[G]$ is already defined over $\Q$, and all the basic invariants of $W$ are contained in $\Q$. As was mentioned in Proposition \ref{prop:Galois}, this only means that $^{\nu}W\cong W$. 
We show here that the cocycle deformation $W$ is, however, not definable over $\Q$.
For this, we begin by writing $W$ explicitly. As an algebra, 
\begin{equation}W = e_1K^{\alpha}F\oplus e_gK^{\nu(\alpha)}F\oplus e_{g^2}K^{\alpha}F\oplus e_{g^3}K^{\nu(\alpha)}F,\end{equation} the action of $g$ is given by
\begin{equation}e_1U_{f_1} + e_gU_{f_2} + e_{g^2}U_{f_3}+ e_{g^3}U_{f_4}\mapsto e_1 U_{g(f_4)} + e_g U_{g(f_1)} + e_{g^2}U_{g(f_2)} + e_{g^3}U_{g(f_3)}.\end{equation}
The action of $f\in F$ is given by conjugation with $e_1U_f + e_gU_f + e_{g^2}U_f + e_{g^3}U_f$.
Assume now that $W$ has a form $W_{\Q}$ defined over the rational numbers. For convenience, we will think of $W_{\Q}$ as a subspace of $W$.
The algebra $W_{\Q}$ is the direct sum of 1, 2 or 4 simple algebras over $\Q$.
Since the action of conjugation by $x$ and by $y$ is defined in $W_{\Q}$, and since this action fixes the centre pointwise, Skolem-Noether Theorem implies, by considering this action on the different direct summands of $W_{\Q}$, that
the action of $x$ is given by conjugation by an element $V_x\in W_{\Q}$, and the action of $y$ is given by conjugation by an element $V_y\in W_{\Q}$.
It must hold that $V_x = a_1e_1U_x + a_ge_gU_x + a_{g^2}e_{g^2}U_x + a_{g^3}e_{g^3}U_x$ and $V_y = b_1e_1U_y + b_ge_gU_y + b_{g^2}e_{g^2}U_y + b_{g^3}e_{g^3}U_y$
We then get that $z=V_xV_yV_x^{-1}V_y^{-1} = \zeta e_1 + \zeta^2 e_g + \zeta e_{g^2} + \zeta^2 e_{g^3}\in W_{\Q}$.
The last element is contained in the centre $Z$ of $W_{\Q}$. We see that $dim_{\Q}Z=4$, since the center of $W$ over $K$ is four dimensional, and the dimension of the centre is stable under field extension.
We also see that $Z$ contains a subalgebra of dimension 2 
which is isomorphic with the field $\Q(\zeta)$ (this is the subalgebra generated by the element $z$).
Moreover, there exists an automorphism $g$ of order 4 such that $Z^g=\Q$ (since taking dimensions of fixed subspaces is also stable under field extension). 
The element $g^2$ fixes $\Q(z)$ pointwise. Since we know that $Z^{g^2}$ has dimension 2 (again, since taking dimensions is stable under field extensions),
we see that $Z^{g^2}= \Q(z)$. The centre $Z$ is either a field or the direct sum of two fields. If it is a field, it will be a Galois extension of $\Q$ of order 4 which contains $\Q(\zeta)$ as a subfield.
However, one can prove directly that such an extension does not exist. 

It follows that $Z$ must be isomorphic with $\Q(\zeta)\oplus \Q(\zeta)$. 
The action of the element $g$ then necessarily sends $(1,0)$ to $(0,1)$, $(0,1)$ to $(1,0)$, $(\zeta,0)$ to either $(0,\zeta)$ or $(0,\zeta^2)$ and $(0,\zeta)$ to either $(\zeta,0)$ or $(\zeta^2,0)$. 
But by checking all 4 cases we see that in none of these options it holds that $Z^{g^2} \cong \Q(\zeta)$.
This is a contradiction, and therefore the form $W_{\Q}$ does not exist. This shows that unlike in the case of group algebras (see \cite{AHN}), the invariants of the 2-cocycle are not always enough in order to define the 2-cocycle.

\subsection{Taft Hopf algebras}
We turn now to examples arising from non-semisimple Hopf algebras.
We consider the Taft Hopf algebra $H_n$. This algebra has the following presentation:
\begin{equation} H_n = K\langle g,x\rangle/(g^n-1,x^n,gxg^{-1}-\zeta x)\end{equation}  where $\zeta\in K$ is a primitive $n$-th root of unity.
The comultiplication in this algebra is given on the generators by the rules $\Delta(g) = g\ot g$ and $\Delta(x) = x\ot 1 + g\ot x$. 
The classification of 2-cocycles for this Hopf algebras are known (see \cite{Masuoka} and the examples in \cite{Meir1}).
We will follow now some of the ideas of \cite{Meir1} and describe this problem as a problem in invariant theory.
To do so, we will begin by taking a cocycle deformation of $H_n$, and analyze its structure. We then use this analysis to reduce the acting group from $GL_{n^2}$ to $GL_1$, and study the invariants under this group.

To do so, we begin by recalling some of the calculations done in Lemma 13.1 of \cite{Meir1}. 
We begin by considering the element $\gamma\in H^*$ given by $\gamma(g^ix^j) = \delta_{j,0}\zeta^i$. 
As can easily be seen, this element is a group like element of $H^*$. 
We also define $\xi\in H^*$ by $\xi(g^ix^j) = \delta_{j,1}$. 
This element satisfies the equation $\Delta(\xi) = 1\ot \xi + \xi\ot\gamma^{-1}$.
Lastly, we define $\ti{T}_g = m\tau T_g:W\to W$. Under the isomorphism $W\cong \allH_n$ this map is given by conjugation by the element $g$ inside $\allH_n$. 

A direct calculation shows that the maps $T_g$ and $A_{\gamma}$ commute with one another. This follows from the fact that the corresponding elements commute inside the Drinfeld double $D(H_n)$ of $H_n$, 
and that $W$ is a representation of $D(H_n)$, as was explained in Section \ref{sec:Prelim}.

We thus see that $W$ is a representation of the abelian group $\Z/n\times \Z/n$. We let $W_{i,j}\subseteq W$ to be the subspace of $W$ upon which $\gamma$ acts by $\zeta^i$ and $T_g$ by $\zeta^j$.
We have a direct sum decomposition $W=\oplus_{i,j}W_{i,j}$. 
Using again the $D(H_n)$-representation structure on $W$, we see that $\xi(W_{i,j})\subseteq W_{i+1,j-1}$ where the indices are taken modulo $n$. 

Using now the presentation $W=\allH_n$, we see that the kernel of $\xi$ has a basis given by $g^i$ for $i=0,\ldots, n-1$. But $g^i\in W_{i,0}$. It follows that the restriction of $\xi$ to $W_{i,j}$ for $j\neq 0$ is injective.
By dimension considerations it follows now that $dim(W_{i,j}) = 1$ for every $i,j=0,\ldots n-1$. It then follows that if $j\neq 0$ then $\xi:W_{i,j}\to W_{i+1,j-1}$ is a linear isomorphism. 
Notice that the result about the dimensions of $W_{i,j}$ is independent of the particular cocycle.

The unit of $W$ must be an element of $W_{0,0}$. 
Let now $t\in W_{n-1,1}$ be an element which satisfies $\xi(t)=1$, and let $\ti{g}\in W_{1,0}$ be some nonzero element.
We claim that for every $i,j=0,\ldots, n-1$ the element $\ti{g}^it^j$ spans the one dimensional space $W_{i-j,j}$.
For this, we first notice that $\ti{g}$ is invertible. This follows from considering the restriction of the cocycle $\alpha$ to the sub-Hopf algebra of $H_n$ generated by the group like element $g$.
It thus remains to prove that $t^j\neq 0$ for $j<n$. We prove this by induction on $j$. We first calculate 
$A_{\xi}(t^j) = A_{\xi}(t\cdot t^{j-1}) = t\cdot A_{\xi}(t^{j-1}) + 1\cdot \zeta^{1-j}t^{j-1}$. It follows now by induction that $A_{\xi}(t^j) = (1+\zeta^{-1}+\cdots+\zeta^{1-j})t^{j-1}$,
so if we assume by induction that $t^{j-1}\neq 0$ it follows that $t^j\neq 0$ as well.

This already gives us a description of almost all the structure constants of $W$. 
Indeed, $W$ has a basis given by the elements $\ti{g}^it^j$. Using the associativity of $W$, and the fact that the action of $\ti{g}$ on $t$ by conjugation is given by $\ti{g}t\ti{g}^{-1} = \zeta t$ (because $t\in W_{n-1,1}$), 
we get that the multiplication of two basis elements is given explicitly by the formula 
\begin{equation} (\ti{g}^it^j)\cdot (\ti{g}^kt^l) = \begin{cases}
                            \zeta^{-jk}\ti{g}^{i+k}t^{j+l} \text{ if } i+k,j+l<n, \\
                            \zeta^{-jk}\ti{g}^n\ti{g}^{i+k-n}t^{j+l} \text{ if } i+k\geq n,j+l<n, \\
                            \zeta^{-jk}t^n\ti{g}^{i+k}t^{j+l-n} \text{ if } i+k<n,j+l\geq n, \\
                            \zeta^{-jk}\ti{g}^nt^n \ti{g}^{i+k-n}t^{j+l-n} \text{ if } i+k,j+l<n. \\
                           \end{cases} \end{equation} 
where we have used the fact that $\ti{g}^n,t^n\in W_{0,0} = span\{1\}$ are central elements. We will denote these elements by $\ti{g}^n=a$ and $t^n=b$.

Let $\{w_{i,j}\}_{i,j}$ be a basis for $W$. We now summarize the above discussion by using the Group-Reduction-Proposition (Proposition \ref{prop:Reduction}). 
Let \begin{equation} Y'=\{(m,T,A)| w_{i,j} = w_{1,0}^iw_{n-1,1}^j,A_{\xi}(w_{n-1,1})=1,\end{equation}  \begin{equation} \ti{T}_g(w_{i,j})=\zeta^jw_{i,j},A_{\gamma}(w_{i,j}) = \zeta^{i-j}w_{i,j}\}\subseteq Y.\end{equation} 
The above discussion shows that every orbit of $GL(W)$ in $Y$ intersects $Y'$. This is because we have shown that we can always find an orbit of $W$ for which the structure constants has a specific form. Since the action of the grop $GL(W)$ is esentially the base change operation, this implies that $Y'$ intersect all the orbits in $Y$.

Consider now the subgroup $N:=GL_1\subseteq GL(W)$ which acts on $W$ by the formula $\lambda\cdot w_{i,j} = \lambda^iw_{i,j}$.
The subgroup $N$ acts on $Y'$, and for every $GL(W)$-orbit $\Ow$ it holds that $\Ow\cap Y'$ is a single $N$-orbit.
This follows from the fact that the only liberty we have in the choice of the basis is by choosing $\lambda \ti{g}$ instead of $\ti{g}$.
The Group-Reduction-Lemma can thus be used here. All the structure constants can be expressed as polynomials in $a^{\pm 1},b$ as explained above: 
This is true for the structure constants of the multiplication by the above formula. 
For the structure constants of $A$, we have $A_{\gamma}(w_{i,j}) = \zeta^{i-j}w_{i,j}$ and $A_{\xi}(w_{i,j}) = (1+\zeta^{-1}+\ldots+\zeta^{1-j})h^it^{j-1}$,
and this is enough to describe $A$ because $\gamma$ and $\xi$ generate $H^*$.
We can thus also describe $T$ by using $a^{\pm 1}$ and $b$, since $T$ is the inverse of a linear map constructed from $A$ and $m$ 
(a priori this only tells us that we can describe the entries of $T$ by rational functions of $a$ and of $b$, but a direct verification shows that polynomials in $a^{\pm 1}$ and $b$ are enough).
By checking directly all the equations the structure $W$ should satisfy, we see that for every value of $(a,b)$ with $a\neq 0$ we will get an algebra which is a cocycle deformation for $H_n$.

We thus reach the conclusion that $K[Y']=K[a^{\pm 1},b]$. The action of $N$ on this ring is given by $\lambda\cdot a = \lambda^na$, $\lambda\cdot b = b$.
The ring of invariants is thus $K[b]$. 
We thus get here a concrete description of the moduli space of 2-cocycles on $H_n$ as the affine line. 

This result has been proven also in \cite{Meir1}. However, the proof there was different: we have also showed that over an algebraically closed field the equivalence class of a cocycle is given by a scalar $b\in K$.
What was less clear was the fact that for every $b$ we really get such an algebra. The proof there was by constructing such an algebra explicitly, using a crossed product construction.
This part of the proof appears also here, where we verified that for every value of $(a,b)$ with $a\neq 0$ we get a cocycle deformation of $H_n$ (by checking the equations for associativity and the other axioms).
In the next examples we will have to make similar considerations. We will use again the Group-Reduction Proposition in order to reduce the acting group to a relatively small group.
This will enable us to describe explicitly a generating set for the invariants. The problem is that it is not clear a-priori what will be the relations that these invariants should satisfy.
We will then use Lemma \ref{lemma:Minv} to show that for every set of invariants we get a cocycle deformation (and so, the generating sets we find will freely generate a polynomial algebra).

\subsection{Hopf algebras arising as Bosonizations of finite non-abelian groups}\label{subsec:S3}
We will consider now examples of cocycle deformations of Hopf algebras of the form $H=H_0\# R$ where $H_0$ is a semisimple Hopf algebra and $R\in\! _{H_0}^{H_0}\YD$ is the Nichols algebra $\B(V)$ of some object $V\in \!_{H_0}^{H_0}\YD$. 
We will consider here the case where $H_0=KG$ or $H_0=K[G]$ for a finite group $G$.
In both cases the category of Yetter Drinfeld modules is the same. An object in $_{H_0}^{H_0}\YD$ is a representation $V$ of $G$, which is also $G$ graded, in such a way that $g(V_h) = V_{ghg^{-1}}$ for every $g,h\in G$. 
We consider now the case where $G=S_3$ and $V=Ind_{\langle(12)\rangle}^{S_3}sign$, where $sign$ is the sign representation of $\langle(12)\rangle$. 
In the terminology of \cite{IM} this representation is denoted $(\Ow^3_2,-1)$. 
Written explicitly, $V=span\{a,b,c\}$ where 
$$a=1\ot 1$$
$$b= (123) \ot 1\text{ and}$$
$$c = (132)\ot 1.$$
The degrees of the elements of $V$ is given by 
$$\text{deg}(a) = (12)$$
$$\text{deg}(b) = (23)$$
$$\text{deg}(c) = (13)$$
The action of $(12)$ on the basis elements of $V$ is given by $$a\mapsto -a, b\mapsto -c, c\mapsto -b,$$ and the action of $(123)$ is given by $$a\mapsto b\mapsto c\mapsto a.$$
The structure of the Nichols algebra $\B(V)$ was studied in \cite{MS}. It is given explicitly as $\T(V)/I$, where $I$ is the two-sided ideal generated by the degree 2 elements $$a^2,b^2,c^2,ab+bc+ca\text{ and } ac+cb+ba.$$
This algebra is known as $FK_3$, the Fomin-Kirillov algebra. See also \cite{HecVen2}. The algebra $\B(V)$ is graded, and the dimensions of its homogeneous components are 1,3,4,3,1 (where $\dim \B(V)_0=1$ and so on).

We will study now cocycle deformations on the Hopf algebras $\B(V)\#KS_3$ and $\B(V)\#K[S_3]$. 
Notice that the first Hopf algebra is pointed (that is- all its irreducible comodules are one dimensional), while the second one is not, because the group $S_3$ is not abelian. Both Hopf algebras are of dimension 72.
See also \cite{AndVay1} and \cite{AndVay2} for a deeper study of these Hopf algebras.\\

\subsection{The Hopf algebra $H=\B(V)\#KS_3$.}
We have $H=KS_3\otimes \B(V)$ as vector spaces. The algebra structure is the bicrossed product algebra structure, where the group $S_3$ acts on $\B(V)$ by conjugation (see \cite{AndSch}). The elements of $S_3$ are group-like elements, and the elements of $V$ satisfy the coproduct formula
$$\Delta(a) = a\ot 1 + (12)\ot a$$
\begin{equation}\Delta(b) = b\ot 1 + (23)\ot b\end{equation}
$$
\Delta(c) = c\ot 1 + (13)\ot c
$$

 Analysing the equations for the variety $X_H$ will be too complicated.
Instead, we will use again the group reduction technique. For this we will begin by finding a specific form for any cocycle deformation of the Hopf algebra $H_1$.

Notice first that the Hopf algebra $H$ is also graded, where $$H_i = \T(V)_i\cdot KS_3.$$ For every $j$ it then holds that $$\oplus_{i\leq j} H_i$$ is a subcoalgebra of $H$.
Let now $W=\allH$ for some 2-cocycle $\alpha$. We begin by analysing $W$ similar to the way we proceeded in the previous examples. In the language of \cite{Meir1}, we study the fundamental category of $W$.

Whenever it will be convenient for us, we will use the isomorphism $W\cong H$ of right $H$-comodules.
We recall also the fact that $$T(1\ot h) = \tiS(h_1)\ot h_2\in W\ot W.$$
We start with analysing the comodule structure of $W$.
For every $g\in S_3$ we write $$W_g = \rho^{-1}(W\ot Kg).$$ This is a subspace of $W$ of dimension 1 (by using the isomorphism of comodules),
and the direct sum $$W_{S_3}=\oplus_{g\in S_3} W_g$$ will give us an algebra which is isomorphic to $K^{\beta}S_3$ for some $$[\beta]\in H^2(S_3,K^{\times}).$$ 

Next, we consider the subspace $$H_a=span\{1,a\}\subseteq H.$$ It holds that $\Delta^{-1}(H\ot H_a) = H_a.$ 
It then holds that $$W_a:=\rho^{-1}(W\ot H_a)\subseteq W$$ is a subspace of dimension 2.
Pick a nonzero element $w_{(12)}\in W_{(12)}$. We know that this element is invertible because the restriction of the cocycle $\alpha$ to $KS_3$ is also invertible.
Write $$c_{(12)}:W\to W$$ for conjugation by the element $w_{(12)}$. Notice that $c_{(12)}$ does not depend on the choice of $w_{(12)}$, and has order 2, since $w_{(12)}^2\in span\{1\}$.
Since $(12)a(12)=-a$ in $H$, it holds that $c_{(12)}(W_a)=W_a$. We write $\{1,w\}$ for a basis of $W_a$. It holds that $c_{(12)}(1)=1$. If $$\rho(w) = w_1\ot 1 + w_2\ot a,$$
then from the coassociativity and counitality of the coaction of $H$ we get that $w_1=w$ and $\rho(w_2) = w_2\ot (12)$.
This implies that $w_2$ is proportional to $w_{(12)}$, and therefore $c_{(12)}(w_2) = w_2$.
Using the fact that $\rho$ is an algebra morphism, we get $$\rho(c_{(12)}(w)) = c_{(12)}(w_1)\ot (12)(12)+ c_{(12)}(w_2)\ot (12)a(12) = $$ $$c_{(12)}(w_1)\ot 1 - w_2\ot a.$$
We thus have that \begin{equation}\rho((c_{(12)} + Id_W)(w)) = (c_{(12)} + Id_W)(w)\ot 1,\end{equation} and therefore $$(c_{(12)}+Id_W)(w)\in span\{1\}.$$
It follows that $$Im(c_{(12)}+Id_W))|_{W_a} = span\{1\}.$$ Since $dim W_a=2$ this means that there exists a unique vector (up to a nonzero scalar) $w_a\in W_a$ such that $c_{(12)}(w_a) = -w_a$.
This follows by considering the representation of the group $\Z/2$ (given by the conjugation action of $c_{(12)}$) on $W_a$ and its decomposition to isotypic components.

We next choose a non-zero element $w'_{(23)}\in W_{(23)}$. This element is again invertible since the cocycle $\alpha$ is invertible, and we consider the element $$w_{(123)} = w'^{-1}_{(23)}c_{(12)}(w'_{(23)})\in W_{(123)}.$$
Notice that this element is canonically defined, and does not depend on the choice of the elements $w_{(12)}$ and $w'_{(23)}$. 
Moreover, since $$H^2(S_3,K^{\times}) = 0$$ one can prove that $w_{(123)}^3=w_1$. We then define $$w_{(23)} = w_{(12)}w_{(123)} , w_{(13)} = w_{(12)}w_{(123)}^2,$$
\begin{equation}w_b = w_{(123)}w_aw_{(123)}^{-1}\text{ and } w_c = w_{(123)}^{-1}w_aw_{(123)}.\end{equation}
We can write $$\rho(w_a) = w_a\ot 1 + \nu w_{(12)}\ot a$$ for some scalar $\nu$. The element $w_a$ is not central, and $\nu$ is therefore not zero. 
We can then rescale $w_a$ to assume that $\nu=1$. Using the definition of $w_b$ and $w_c$ and the multiplicativity of $\rho$ we get that 
$$\rho(w_{b}) = w_{b}\ot 1 + w_{(23)}\ot b \text{ and }$$ $$\rho(w_{c}) = w_{c}\ot 1 + w_{(13)}\ot c.$$
We claim the following:
\begin{lemma}\label{lemma:productsbasis}
 The set of products $$\{1,w_a,w_b,w_c,w_aw_b,w_aw_c,w_bw_a,w_bw_c,$$ $$w_aw_bw_a,w_aw_bw_c,w_bw_aw_c,w_aw_bw_aw_c\}\cdot \{w_g|g\in S_3\}$$ is a basis for $W$
\end{lemma}
\begin{proof}
We will use the multiplicativity of the map $\rho:W\to W\ot H$ together with the fact that the corresponding set of products in $H$ is a basis for $H$.
Since the dimension of $W$ agrees with the size of the set we have, it is enough to prove that this set is linearly independent.
Assume that we have a linear relation of the form \begin{equation}w_aw_bw_aw_c\cdot X + (\text{ terms of degree $\leq 3$} ) = 0\end{equation} where $X\in span\{w_g|g\in S_3\}$.
Apply the map $$W\stackrel{\rho}{\to}W\ot H\to W\ot H_4$$ where the second map is given by the projection $H\to H_4$.
If $X= \sum_{g\in S_3} t_gw_g$ then we get that $$\sum_{g\in S_3}t_gw_{(12)}w_{(23)}w_{(12)}w_{(13)}w_g\ot abacg=0.$$
But this implies that for every $g\in S_3$ we have $$t_gw_{(12)}w_{(23)}w_{(12)}w_{(13)}w_g = 0,$$ which implies that $X=0$ (since all the $w_g$ are invertible).
We then write the above linear combination as a linear combination of terms of degree $\leq 3$, and continue in a similar fashion.
\end{proof}

This gives us a very big reduction of the acting group. 
Indeed, by the last lemma we see that once we have chosen the element $w_{(12)}\in W_{(12)}$ we have a canonically defined basis for $W$. 
We write $$w_{(12)}^2 = c_0, w_a^2 = \lambda_0 \text{ and } w_aw_b + w_bw_c + w_cw_a=\mu_0.$$
Notice that due to the nature of defining $w_b$ and $w_c$ as conjugates of $w_a$, it holds that $$w_b^2 = w_c^2 = \lambda_0, \text{ and }
w_bw_a+ w_aw_c+w_cw_b = \mu_0.$$

We can now write specifically a subvariety $Y'\subseteq Y$ for the Group-Reduction-Lemma in the following way: Let $\{w_1,\ldots w_{72}\}$ be a basis for $W$, and let $\{q_1,\ldots q_{72}\}$ be an enumeration of the basis from the last lemma.
So for example $q_1=1$, $q_2=w_{(12)},\ldots q_7=w_a,q_8=w_aw_{(12)}$ and so on. 
We know that these basis elements satisfy a large collection of relations arising from the coaction of $H$ and the multiplication, for example $q_8=q_7q_2$, and $\rho(q_2) = q_2\ot (12)$.
Let us denote by $\{R_i(q_j)\}$ the set of all such relations. 
Consider now the subset $Y'= \{(m,T,A)| \forall i\, R_i(w_j)\}$. The last lemma implies that every orbit in $Y$ intersects $Y'$. 
This follows by choosing carefully the basis for $W$ and considering the fact that the action of $\Ga$ on $Y$ is given by base change.
By the above lemma we see that the only freedom in choosing the basis elements above is in choosing $w_{(12)}$. In other words, consider the subgroup $N=GL_1\subseteq GL(W)$.
The specific embedding of $GL_1$ in $GL(W)$ is given in the following way: an element $\nu\in N$ acts diagonally with respect to the above basis.
If $w_{i_1}\cdots w_{i_t}$ is a monomoial of degree $t$ in $w_a,w_b,w_c$ which appears in the basis, and $g\in S_3$ then the action of $\nu$ on $w_{i_1}\cdots w_{i_t}w_g$ is given by the scalar $\nu^{t+\frac{1-sign(g)}{2}}$.
We then see that each $GL(W)$-orbit in $Y$ intersects $Y'$ in exactly one $N$ orbit. The conditions of Proposition \ref{prop:Reduction} then hold.

We can now write explicitly all the structure constants using the scalars $c_0,\lambda_0$ and $\mu_0$.
The action of $\nu\in GL_1$ on these elements is given by $\nu\cdot (c_0,\lambda_0,\mu_0) = (\nu c_0,\nu\lambda_0,\nu\mu_0)$.
Since the scalar $c_0$ is necessarily invertible (since the element $w_{(12)}$ must be invertible) we reach the conclusion that the ring of invariants $K[Y']^N$ is generated here by the
elements $\lambda:=\lambda_0c_0^{-1}$ and $\mu:=\mu_0c_0^{-1}$.

We conclude this result in the following proposition:
\begin{proposition} 
We have an isomorphism $K[Y']^N\cong K[\lambda,\mu]/I$, where $I$ is some ideal.
\end{proposition}
Therefore, to complete the classification of cocycle deformations of $H$, we need to understand the ideal $I$.
In other words- we need to understand what polynomial relations (if any) the elements $\lambda$ and $\mu$ must satisfy.
We will prove here the following (see also \cite{HecVen2} and \cite{IM}):
\begin{proposition}\label{prop:allS3}
The ideal $I$ above is the zero ideal. In other words, the moduli space $X_H$ of cocycle deformations for $H$ is the affine space $\Aa^2$.
\end{proposition}
\begin{proof}
In principle, we can write all the structure constants for $W$ using $\lambda$ and $\mu$, and check that for every value of $(\lambda,\mu)$ we get a cocycle deformation.
Unfortunately, as the dimension of $W$ is quite big (72), this requires some effort.
What we will do instead is to construct, for every pair $(\lambda,\mu)$, such a cocycle deformation. We thus fix now $(\lambda,\mu)\in \Aa^2$.
We begin by defining an algebra $R$ by the presentation:
	$$R = K\langle w_a,w_b,w_c\rangle/(w_a^2 - \lambda, w_b^2-\lambda, w_c^2-\lambda,$$ $$ w_aw_b+w_bw_c+w_cw_a-\mu, w_aw_c+w_cw_b+w_bw_a-\mu).$$
This algebra has an action of the group $S_3$ given on the generators of $S_3$ and on the generators of $R$ by 
$$(12)(w_a) = -w_a, (12)(w_b) = -w_c, (12)w_c = -w_b$$
$$(123)(w_a) = w_b, (123)(w_b) = w_c, (123)w_c = w_a.$$
We define $W$ to be the crossed product algebra 
$$W:=R\rtimes KS_3.$$ We write $w_g$ for the group elements in $S_3$ inside $W$. Notice that $W=0$ if and only if $R=0$. At this level it is still not clear if this is the case or not. 

We define $\rho:W\to W\ot H$ as the algebra map which is defined, on generators, by $$\rho(w_g) = w_g\ot g, \rho(w_a) = w_a\ot 1 + w_{(12)}\ot a,$$
$$\rho(w_b) = w_b\ot 1 + w_{(23)}\ot b, \rho(w_c) = w_c\ot 1 + w_{(13)}\ot c.$$
A direct verification shows that $\rho$ sends the defining relations of $W$,  arising from the relations among the generators of $R$ and from the formulas for the action of $S_3$ on $R$, to zero, and thus defines an algebra morphism. Verification of coassociativity is also direct.
We next define $\ti{T}:H\to W^{op}\ot W$ to be the algebra map which, on the level of generators of $H$ is given by 
$$g\mapsto w_{g^{-1}}\ot w_g, a\mapsto -w_{(12)}w_a\ot 1 + w_{(12)}\ot w_a,$$ 	$$b\mapsto -w_{(23)}w_b\ot 1 + w_{(12)}\ot w_a, c\mapsto -w_{(13)}w_c\ot 1 + w_{(13)}\ot w_c.$$
A direct verification shows that the conditions of Lemma \ref{lemma:Minv} are satisfied, so if we can prove that the algebra $W$ is nonzero, we will be done.

Notice, again, that this is not clear a-priori. Indeed, it is possible that by deforming the original relations into non-homogeneous ones we would have gotten the element 1 inside the ideal of relations, 
and this would imply that $W=0$. However, Proposition \ref{prop:FormofW} tells us that proving that $W\neq 0$ will be enough. 
We shall do so by proving that $R\neq 0$.

In order to do so, it will be enough to show that $R$ has a nonzero quotient.
We proceed as follows: if $\mu=\lambda=0$ we know that this is the case, since we can simply map $R\to K$ by sending $w_a,w_b,w_c$ to $0$.
If $\lambda=0$ but $\mu\neq 0$, then we consider the algebra homomorphism $\phi: R\to \mathrm{M}_2(K)$ given by sending $w_a,w_b\to \begin{pmatrix}0 & 1 \\ 0 & 0\end{pmatrix}$ and $w_c\to \begin{pmatrix} 0 & 0 \\ \mu & 0\end{pmatrix}$.
A direct calculation shows that $\phi$ is an algebra map.

If $\lambda\neq 0$ we define $X = \begin{pmatrix} 1 & 0 \\ 0 & -1\end{pmatrix}\in \mathrm{M}_2(K)$ and
$Y=\begin{pmatrix}0 & 1 \\ 1 & 0\end{pmatrix}\in \mathrm{M}_2(K)$, and define $\phi:R\to \mathrm{M}_2(K)$ by 
$\phi(w_a) = \phi(w_b) = tX$ and $\phi(w_c) = rX+sY$ where $t^2=\lambda$, $r = (\mu-\lambda)/t$ and $s^2 = \lambda-r^2$
(the choice of the scalars is done to ensure that the relations of $R$ will hold in $\mathrm{M}_2(K)$).
Again, this proves that $R\neq 0$ and we are done.
\end{proof}

We denote the last cocycle deformation by $W=W_{\lambda,\mu}$. By applying Lemma \ref{lemma:Schauenburg} we get the following description of the double-twisted Hopf algebra $\alHal$:
\begin{proposition}
Write $W=\allH$ for an appropriate cocycle $\alpha$ on $H$. 
Then $\alHal$ is the Hopf algebra generated by the group like elements of $S_3$ together with the elements $a,b,c$, and which satisfy the relations 
$a^2=b^2=c^2=0$, $ab+bc+ca = \mu(1-(123))$, and $ba+ac+cb = \mu(1-(132))$. The coproduct is the same as in the Hopf algebra $H$,
and the action of the group like elements on $a,b$ and $c$ is the same as in $H$.
\end{proposition}
\begin{remark}
The construction of this algebra was done in \cite{IM} by Garcia Iglesias and Mombelli. The calculation done there also includes checking that a certain algebra is non-zero. 
They do so by using computer software. Checking that the algebra $R$ has a nonzero quotient, as we did here, provides a shorter proof.
\end{remark}
\begin{remark}
Since we have used Proposition \ref{prop:Reduction}, it is possible that $K[Y]^{GL(W)}$ will be a proper subalgebra of $K[Y']^N$. One can prove directly that this is not the case by constructing basic invariants which are equal to $\la$ and to $\mu$. A similar statement holds for the Taft Hopf algebras. 
\end{remark}

\subsection{Non-pointed non-semisimple Hopf algebra over $S_3$}\label{subsec:dualS3}
We turn now to the Hopf algebra $H=\B(V)\# K[S_3]$. This algebra was studied by Andruskiewitsch and Vay in \cite{AndVay1} and \cite{AndVay2}.
They have also classified Hopf algebras arising from this Hopf algebra by deformation, and showed that these algebras are all of the form $\alHal$.
Here we will consider the classification of all cocycle deformations of this algebra.
The braided vector space $V\in\! ^{S_3}_{S_3}\YD$ is the same as before. This algebra is generated by the dual group algebra $K[S_3]$ together with three elements $a,b,c$ which satisfy the same relations as before: 
$a^2=b^2=c^2=ab+bc+ca = ac+cb+ba = 0$. The action of $K[S_3]$ on $\B(V)$, which amounts to the grading by $S_3$ gives us here that $$e_ga = ae_{(12)g}, e_gb = be_{(23)g}, \text{ and } e_gc = ce_{(13)g} \text{ for every }g\in S_3.$$
The comultiplication is more complicated, as it is induced now by the action of $S_3$. It is given by the formulas:
$$\Delta(a) = a\ot 1  + (e_1-e_{12})\ot a + (e_{132}-e_{13})\ot b + (e_{123}-e_{23})\ot c$$
\begin{equation}\Delta(b) = b\ot 1  + (e_1-e_{23})\ot b + (e_{132}-e_{12})\ot c + (e_{123}-e_{13})\ot a\end{equation}
$$\Delta(c) = c\ot 1  + (e_1-e_{13})\ot c + (e_{132}-e_{23})\ot a + (e_{123}-e_{12})\ot b.$$

Let now $W$ be a cocycle deformation of $H$. We would like to analyze, as before, the structure of $W$ and describe the moduli space of all cocycle deformations.
We start now by considering the restriction of $W$ to $K[S_3]$. 
In other words: let us write $W_{S_3}:= \rho^{-1}(W\ot K[S_3])$ where $\rho$ denotes, as before, the coaction of $H$ on $W$.
By using the fact that $W$ is isomorphic with $H$ as an $H$-comodule, we get that $dim(W_{S_3})=6$, and that this is a cocycle deformation of $K[S_3]$
(in cocycle terms, we will just get here the restriction of the cocycle on $H$ to $K[S_3]$).
We know that all cocycle deformations on $K[S_3]$ are trivial. Indeed, this follows from Proposition \ref{prop:dualgroup} above, and the fact that 
the group $S_3$ does not contain any nontrivial subgroups of central type.
We thus fix an isomorphism $\Psi:K[S_3]\to W_{S_3}$ as $K[S_3]$-comodule algebras.

Notice that this isomorphism is not unique: indeed, the set of such isomorphisms is a torsor over the group of automorphisms of $K[S_3]$ as a $K[S_3]$-comodule algebra,
which is the finite group $S_3$. This lack of uniqueness, very similar to the lack of uniqueness in choosing $w_{(12)}$ in the previous example, 
will come into play later when we will determine the invariants.

We consider the isomorphism $\Psi$ from now on as an identification. We thus have a basis for $W_{S_3}$ given by $\{e_g\}_{g\in S_3}$, 
the multiplication is given by $$e_ge_h = \delta_{g,h}e_g$$ and the coaction of $H$ is given by $$\rho(e_g) = \sum_{g_1g_2=g}e_{g_1}\ot e_{g_2}\in W\ot H.$$
Next, we consider the subspace $$W_1:= \rho^{-1}(W\ot span\{1,a,b,c\}).$$ 
By comparing $W$ and $H$ as $H$-comodules we find out that $dim(W_1)=4$. As can easily be seen, the multiplicative unit $1\in W$ is contained in $W_1$.
We now define the projections $$E_h:W\to W$$ \begin{equation}w\mapsto \sum_{g\in S_3} e_g\cdot w \cdot e_{hg}.\end{equation}
These maps are projections because all the elements $e_g$ are pairwise orthogonal idempotents. 
A direct calculation shows that $$\sum_h E_h = Id_W \text{ and that }E_hE_{h'} = \delta_{h,h'}E_h.$$
Notice that for $w\in W_{S_3}$ it holds that $E_h(w) = \delta_{1,h}w$. This follows from the fact that the subalgebra $W_{S_3}$ is commutative.

Using again the isomorphism of $W$ and $H$ as $H$-comodules, we find out that for every $w\in W_1$ it holds that $$\rho(w)-w\ot 1 \in W_{S_3}\ot span\{a,b,c\}.$$
We write $$\rho(w) = w\ot 1 + r_a(w)\ot a + r_b(w)\ot b + r_c(w)\ot c.$$
We use the fact that the map $\rho$ is multiplicative. We calculate $\rho(E_1(w)):$
$$\rho(E_1(w)) = \sum_{g_1g_2=g_3g_4} (e_{g_1}\ot e_{g_2}) \cdot \rho(w)\cdot (e_{g_3}\ot e_{g_4}) = $$ $$
\sum_{g_1g_2=g_3g_4} [e_{g_1}we_{g_3}\ot e_{g_2}e_{g_4} + e_{g_1}r_a(w)e_{g_3}\ot e_{g_2}ae_{g_4} + $$ $$e_{g_1}r_b(w)e_{g_3}\ot e_{g_2}be_{g_4} + e_{g_1}r_c(w)e_{g_3}\ot e_{g_2}ce_{g_4}]=$$  
$$\sum_{g_1,g_2\in G}[e_{g_1}we_{g_1}\ot e_{g_2} + e_{g_1}r_a(w)\ot e_{g_2}ae_{g_2} + $$ \begin{equation}e_{g_1}r_b(w)\ot e_{g_2}be_{g_2} + e_{g_1}r_c(w)\ot e_{g_2}ce_{g_2}]= 
E_1(w)\ot 1.\end{equation}
We have used here the fact that the elements $r_a(w),r_b(w)$ and $r_c(w)$ are contained in the commutative algebra $W_{S_3}$, and that for every $g\in G$ it holds that 
$e_gae_g = e_gbe_g = e_gce_g = 0$ in $H$. The fact that $\rho(E_1(w)) = E_1(w)\ot 1$ implies that $E_1(w)$ is a scalar multiple of 1.
This implies that the image of $E_1$ is exactly $K1$, and the kernel of $E_1$ is therefore 3 dimensional. We denote this kernel by $W_2$.

We consider next the map $r_a:W_1\to W_{S_3}$. This is a linear map, and by comparing again with $H$, we find out that the kernel of this map is spanned by 1.
It follows that the restriction $r_a:W_2\to W_{S_3}$ is injective.
We next define $$w_a = r_a^{-1}(e_1-e_{12}), w_b = r_a^{-1}(e_{123}-e_{13}) \text{ and }w_c = r_a^{-1}(e_{132}-e_{23}).$$
Again, we know that all these elements are contained in the image of $r_a$ by comparing $W$ with $H$.
We have to be a bit careful here, since the map $r_a$ depends on the isomorphism $\Psi$ we chose above.
It can be shown that the above elements will be in the image of $r_a$ for every choice of $\Psi$.
The fact that the elements $e_1-e_{12}$, $e_{123}-e_{13}$ and $e_{132}-e_{23}$ are linearly independent in $W_{S_3}$ implies that the elements $w_a,w_b,w_c\in W_1$ are linearly independent as well.
Considering again the isomorphism between $W$ and $H$ we find out the following explicit formulas for $\rho$:
$$\rho(w_a) = w_a\ot 1  + (e_1-e_{12})\ot a + (e_{132}-e_{13})\ot b + (e_{123}-e_{23})\ot c$$
\begin{equation}\label{eq:Rhos}\rho(w_b) = w_b\ot 1  + (e_1-e_{23})\ot b + (e_{132}-e_{12})\ot c + (e_{123}-e_{13})\ot a\end{equation}
$$\rho(w_c) = w_c\ot 1  + (e_1-e_{13})\ot c + (e_{132}-e_{23})\ot a + (e_{123}-e_{12})\ot b.$$

Using the above formulas for $\rho$ we calculate the projections $E_h$ defined before to the elements $w_a,w_b$ and $w_c$. 
For $h\notin \{(12),(23),(13)\}$ we use the fact that inside $H$ we have $e_xae_{hx}=0$, and similarly for $b$ and $c$. We find out that for such $h$ it holds that
$\rho(E_h(w_a)) = E_h(w_a)\ot 1$. This implies that $E_h(w_a)\in K1$. But then $E_h^2(w_a) = 0$. Using the fact that $E_h$ is a projection, we conclude that $E_h(w_a)=0$.
A similar result holds for $w_b$ and $w_c$.

For the transpositions, we calculate $E_{(23)}(w_a)$, using the commutativity of $W_{S_3}$. 
Similar to the calculation of $E_1(w)$ before, and using the explicit formulas for $\rho$ above, we find out that $$\rho(E_{(23)}(w_a)) = E_{(23)}(w_a)\ot 1.$$ 
By the same argument as before we conclude that $$E_{(23)}(w_a)=0.$$
Similar results hold for $E_{(13)}(w_a), E_{(12)}(w_b), E_{(13)}(w_b), E_{(12)}(w_c),$ and $E_{(13)}(w_c)$.
If we write $Im(E_h) = W^h$ then this gives us a direct sum decomposition $W=\oplus_h W^h$.
An easy calculation then shows that $$W^{h_1}\cdot W^{h_2}\subseteq W^{h_1h_2}\text{ for every } h_1,h_2\in S_3.$$

The proof of the following lemma is very similar to the proof of Lemma \ref{lemma:productsbasis}, and we omit it.
\begin{lemma}
The set of products $$\{1,w_a,w_b,w_c,w_aw_b,w_aw_c,w_bw_a,w_bw_c,$$ $$w_aw_bw_a,w_aw_bw_c,w_bw_aw_c,w_aw_bw_aw_c\}\cdot \{e_g|g\in S_3\}$$ is a basis for $W$
\end{lemma}
We almost have the entire structure of $W$ in our hands now. The last thing that we need to do is to find out how the homogeneous defining relations in $\B(V)$ are deformed in $W$.

Consider the element $w_a^2\in W$. We calculate $\rho(w_a^2):$
$$\rho(w_a^2) = \rho(w_a)^2 = (w_a\ot 1 + (e_1-e_{12})\ot a + (e_{132}-e_{13})\ot b + (e_{123}-e_{23})\ot c)^2 = $$
$$ w_a^2\ot 1 + (a(e_1-e_{12}) + (e_1-e_{12})a)\ot a + (a(e_{132}-e_{13}) + (e_{132}-e_{13})a)\ot b + $$ \begin{equation}(a(e_{123}-e_{23}) + (e_{123}-e_{23})a)\ot c = w_a^2\ot 1.\end{equation}
We have used here the fact that in $H$ it holds that $a^2=b^2=c^2=0$, and the fact that $E_{(12)}(w_a)=w_a$, which implies that all the anti-commutators vanish.
This implies that $w_a^2\in K1$ and we write $$w_a^2=\la_a.$$ Similarly, $w_b^2\in K1$ and $w_c^2\in K1$ and we write $$w_b^2=\la_b\text{ and }w_c^2 = \la_c.$$

We turn next to the second type of relations. A similar calculation reveals the fact that 
$$\rho(w_aw_b + w_bw_c + w_cw_a) = w_aw_b + w_bw_c + w_cw_a\ot 1\text{ and }$$
\begin{equation}\rho(w_aw_c + w_cw_b + w_bw_a) = w_aw_c + w_cw_b + w_bw_a\ot 1.\end{equation} This implies that these elements are also scalar multiples of 1.
But $1\in W^1$ and $w_aw_b+w_bw_c + w_cw_a\in W^{(123)}$, and so we get the relations
$$w_aw_b + w_bw_c + w_cw_a=w_aw_c + w_cw_b + w_bw_a=0.$$ Notice that the above problem does not arise for $w_a^2$ since $w_a^2\in W^1$.

The last relations already give us a complete description of $W$. We summarize it here:
\begin{proposition}\label{prop:relations}
The algebra $W$ is generated by the elements $\{e_g|g\in S_3\}, w_a,w_b,w_c$. 
These elements satisfy the following relations:
$$e_ge_h = \delta_{g,h}e_g, \sum_g e_g=1$$
$$e_gw_a = w_ae_{g(12)}, e_gw_b = w_be_{g(23)}, e_gw_c = w_ce_{g(23)}$$
$$w_a^2 = \la_a, w_b^2 = \la_b, w_c^2 = \la_c$$
$$w_aw_b+w_bw_c+w_cw_a = w_aw_c+w_cw_b+w_bw_a = 0$$
\end{proposition}

We know the map $\rho$ on the generators $e_g$ and $w_a,w_b,w_c$. 
Indeed, we have $\rho(e_g) = \sum_{g'\in S_3}e_{g'}\ot e_{g'^{-1}g}$ and $\rho(w_x)$ is given by Equation \ref{eq:Rhos}.
We claim the following:
\begin{proposition}
For every value of $(\la_a,\la_b,\la_c)$ there exists a cocycle deformation algebra $W$ in which $w_a^2=\la_a$, $w_b^2=\la_b$ and $w_c^2=\la_c$.
\end{proposition}
\begin{proof} (see also Theorem 3.2. of \cite{AnG})
As in the Example of Subsection \ref{subsec:S3}, it will be enough to prove that the algebra $R$ generated by $w_a,w_b$ and $w_c$ modulo the relations $w_aw_b+w_bw_c+w_cw_a=w_aw_c+w_cw_b+w_bw_a=0, w_a^2=\la_a, w_b^2=\la_b,$ and $w_c^2=\la_c$
is nonzero. In case $\la_a=\la_b=\la_c$ this follows from Proposition \ref{prop:allS3}, by considering the case where $\la=\la_a=\la_b=\la_c$ and $\mu=0$.

Assume then that $\#\{\la_a,\la_b,\la_c\}\geq 2$. Without loss of generality we assume that $\la_b\neq \la_c$. 
We begin by making some calculations inside the algebra $R$. We have $$w_a\cdot w_c^2= (w_aw_c)w_c = -w_bw_aw_c-w_cw_bw_c. $$
On the other hand, we have that $$w_bw_aw_c = -w_bw_bw_a - w_bw_cw_b$$ and so we get $$w_a\cdot w_c^2 = w_b^2w_a + w_bw_cw_b - w_cw_bw_c.$$
This implies that $$(\la_b-\la_c)w_a = w_cw_bw_c-w_bw_cw_b.$$ By our assumption this means that $w_a$ can be written as a noncommutative polynomial in $w_b$ and $w_c$, and so the algebra $R$ is generated by the elements $w_b$ and $w_c$.
We re-write the defining relations of $R$ as noncommutative polynomials in $w_b$ and $w_c$. For this, we write $(\la_b-\la_c)^{-1}=x$. 
The cyclic relation $w_aw_b+w_bw_c + w_cw_a=0$ then becomes, modulo the quadratic relations for $w_b$ and $w_c$:
$$x(w_cw_bw_c-w_bw_cw_b)w_b + w_bw_c + xw_c(w_cw_bw_c-w_bw_cw_b) = $$ \begin{equation}x(w_cw_bw_cw_b-w_bw_c(\la_b-\la_c)-w_cw_cw_cw_b) + w_bw_c = 0.\end{equation}
In other words, it becomes redundant. The same happens for the other cyclic relation.
We are left with the algebra generated by the elements $w_b$, $w_c$ and defined by the relations
$$w_b^2=\la_b,w_c^2=\la_c$$
\begin{equation}x^2(w_cw_bw_c-w_bw_cw_b)^2=\la_a\end{equation}
We open the parenthesis in the last relation and we get 
\begin{equation}x^2(\la_c^2\la_b + \la_b^2\la_c - (w_cw_b)^3 - (w_bw_c)^3)=\la_a\end{equation}
Since $\la_b\neq \la_c$ at least one of them is different from zero. Assume without loss of generality that $\la_b\neq 0$.
This implies that $w_b$ is invertible.
We write now $w_cw_b = v_1$, $w_bw_c=v_2$. Then the algebra $R$ is generated by the elements $w_b^{\pm 1}$, $v_1$ and $v_2$, and has the defining relations 
$$w_b^2=\la_b, v_1v_2 = \la_c, w_bv_1w_b^{-1}=v_2,w_bv_2w_b^{-1}=v_1\text{ and } $$ 
\begin{equation}\label{eq:Rel1} v_1^3+v_2^3 = y\text{ where }y = \la_ax^{-2}\la_c\la_b(\la_b+\la_c).\end{equation}
It follows that the subring $S\subseteq R$ generated by $v_1$ and $v_2$ is commutative. 
Moreover, since conjugation by $w_b$ stabilizes $S$ and stabilizes the relations, we see that the ring $R$ is the crossed product $$R=S*\langle w_b\rangle,$$
and that the ring $S$ can be represented by the generators $v_1,v_2$ and the relations $v_1v_2=v_2v_1= \la_c$ and $v_1^3+v_2^3=y$.
In order to prove that $R\neq 0$ it will thus be enough to prove that $S\neq 0$.

To prove this, consider first the case $\la_c=0$. We then have that $$S/(v_2) = K[v_1]/(v_1^3-z)\neq 0 $$ where $z$ is an appropriate scalar.
In case $\la_c\neq 0$ then $v_2$ is the inverse of $v_1$ up to a nonzero scalar, and we can write $$S = K[v_1^{\pm 1}]/(f)\neq 0$$ where $f$ is a polynomial of degree 6.
In any case, the ring $S$ is not the zero ring and we are done.
\end{proof}
So we see that every cocycle deformation of $H$ arises from a triple $(\la_a,\la_b,\la_c)$, and that every such triple gives a cocycle deformation.
However, there is no one to one correspondence between such triples and cocycle deformations of $H$. The reason for this is that the triple $(\la_a,\la_b,\la_c)$ 
depends on the choice of the isomorphism $\Psi$ we made before, and $\Psi$ is defined only up to an element of $S_3$.

The resulting action of the group $S_3$ on the variety $\Aa^3$ is given by permuting the coordinates $(\la_a,\la_b,\la_c)$.
This implies that the set of cocycle deformations is in a one to one correspondence with $\Aa^3/S_3$. 
Since the group $S_3$ is finite it is reductive, and for obvious reasons all the orbits for this action are closed.
We have $K[\Aa^3/S_3] = K[\Aa^3]^{S_3} = K[c_1,c_2,c_3]$ where $$c_1= \la_a+\la_b+\la_c, c_2 = \la_a\la_b + \la_b\la_c + \la_c\la_a, c_3=\la_a\la_b\la_c.s$$
In other words, the moduli space of cocycle deformation is again isomorphic with $\Aa^3$.
\begin{remark}
 It is also possible to receive this result by using reduction of the acting group to $S_3$. The result will be the same.
\end{remark}

Using again Lemma \ref{lemma:Schauenburg}, we can describe also the Hopf algebra $\alHal$. 
This Hopf algebra is again generated by $e_g,a,b,c$ with the same coproduct as before. 
The multiplication remains almost the same, except for the relations $a^2=b^2=c^2=0$, which deform to 
$$a^2 = (\la_a-\la_b)(e_{13}+e_{132}) + (\la_a-\la_c)(e_{23}+e_{123})$$
$$b^2 = (\la_b-\la_c)(e_{12}+e_{132}) + (\la_b-\la_a)(e_{13}+e_{123})$$
$$c^2 = (\la_c-\la_a)(e_{23}+e_{132}) + (\la_c-\la_b)(e_{12}+e_{123}).$$
These Hopf algebras were also described in \cite{AndVay1} and \cite{AndVay2}.
Notice that in this case most of the analysis of the cocycle deformation was done ``by hand'', and only at the very final step we have used the action of $S_3$ to describe the moduli space $X_H$ as an affine variety.
This still leads to some geometrical questions, which we will describe in the next section. 

\end{section}
\begin{section}{Some further questions}\label{sec:Questions}
In case the Hopf algebra $H$ is semisimple, it is known by Ocneanu's rigidity that there are only finitely many cocycle deformations.
In the two non-semisimple examples which we have seen here, the space of cocycle deformations were affine spaces.
It is easy to combine these two examples to receive Hopf algebras whose space of cocycle deformations is a disjoint union of affine spaces.
This leads us to the following conjecture:
\begin{conjecture}
 Let $H$ be a finite dimensional pointed Hopf algebra. The moduli space of 2-cocycles up to equivalence is then the disjoint union of affine spaces.
\end{conjecture}

In the example in Subsection \ref{subsec:S3} we have had a canonical way to construct the cocycle deformation out of the invariants.
Indeed, one sets $c_0=1$ and receives a cocycle deformation for which all the structure constants are polynomials in the invariants $\mu,\la$.
We thus get a vector bundle of cocycle deformations over the space $X_H$. In the terminology of \cite{HM}, we get a \textit{fine} moduli space of cocycle deformations.
\begin{question}
Let $H$ be a finite dimensional Hopf algebra. Is there a vector bundle of cocycle deformations over $X_H$, similar to the previous example?
\end{question}
For the Hopf algebra which appears in Subsection \ref{subsec:dualS3}, the quotient by the action of $S_3$ makes it unclear how can one construct such a vector bundle.
We do not know the answer to that question even for that case.

\end{section}
\begin{section}*{Acknowledgments}
The author was supported by the Research Training Group 1670, ``Mathematics Inspired by String Theory and Quantum Field Theory''.
\end{section}

\end{document}